\documentclass[11pt]{article}
\usepackage[width=16cm,height=21cm]{geometry}
\usepackage[colorlinks,allcolors=blue]{hyperref}
\usepackage[utf8]{inputenc}
\usepackage{amsmath,amssymb,enumitem}
\usepackage{cite}
\newcommand{\bB}{\mathbb{B}}
\newcommand{\bC}{\mathbb{C}}
\newcommand{\bD}{\mathbb{D}}

\newcommand{\bN}{\mathbb{N}}
\newcommand{\bR}{\mathbb{R}}

\newcommand{\bZ}{\mathbb{Z}}
\newcommand{\bT}{\mathbb{T}}
\newcommand{\cA}{\mathcal{A}}
\newcommand{\cB}{\mathcal{B}}

\newcommand{\cG}{\mathcal{G}}
\newcommand{\cH}{\mathcal{H}}
\newcommand{\cK}{\mathcal{K}}
\newcommand{\cM}{\mathcal{M}}
\newcommand{\cP}{\mathcal{P}}
\newcommand{\cS}{\mathcal{S}}

\newcommand{\cVT}{\mathcal{VT}}
\newcommand{\cV}{\mathcal{V}}
\newcommand{\cW}{\mathcal{W}}
\newcommand{\cX}{\mathcal{X}}
\newcommand{\cZ}{\mathcal{Z}}
\newcommand{\hG}{\widehat{G}}
\newcommand{\hH}{{\widehat{H}}}
\newcommand{\hP}{{\widehat{P}}}

\newcommand{\tN}{{\widetilde{N}}}
\newcommand{\tR}{{\widetilde{R}}}

\newcommand{\clos}{\operatorname{clos}}
\newcommand{\hnu}{\widehat{\nu}}
\newcommand{\la}{\lambda}

\newcommand{\dif}{\mathrm{d}}

\renewcommand{\Im}{\operatorname{Im}}
\renewcommand{\phi}{\varphi}
\newcommand{\conju}[1]{\overline{#1}}
\newcommand{\imagunit}{\operatorname{i}}
\newcommand{\enumber}{\operatorname{e}}
\newcommand{\linspan}{\operatorname{span}}
\renewcommand{\kappa}{\varkappa}
\newcommand{\Ber}{\operatorname{Ber}}
\newcommand{\WOT}{\operatorname{WOT}}
\newcommand{\esssup}{\operatornamewithlimits{ess\,sup}}

\usepackage{extarrows}
\usepackage{mathtools}
\newcommand{\eqdef}{\coloneqq}
\usepackage{amsthm}
\newtheorem{thm}{Theorem}[section]
\newtheorem{prop}[thm]{Proposition}
\newtheorem{lem}[thm]{Lemma}
\newtheorem{cor}[thm]{Corollary}

\newtheorem{assumption}{Assumption}
\theoremstyle{definition}
\newtheorem{example}[thm]{Example}

\newtheorem{rem}[thm]{Remark}

\newcommand{\myurl}[1]{\href{#1}{#1}}
\newcommand{\doi}[1]{\href{https://doi.org/#1}{DOI: #1}}

\usepackage{tikz}
\usetikzlibrary{trees,arrows,shapes,decorations.pathreplacing}

\usepackage{xcolor}

\author{Crispin Herrera-Ya\~{n}ez, Egor A. Maximenko, Gerardo Ramos-Vazquez}
\title{\vspace{-1.0cm}Translation-invariant operators\\
in reproducing kernel Hilbert spaces}

\begin{document}
\maketitle

\begin{abstract}
Let $G$ be a locally compact abelian group with a Haar measure, and $Y$ be a measure space.
Suppose that $H$ is a reproducing kernel Hilbert space of functions on $G\times Y$,
such that $H$ is naturally embedded into $L^2(G\times Y)$
and is invariant under the translations associated with the elements of $G$.
Under some additional technical assumptions,
we study the W*-algebra $\mathcal{V}$ of translation-invariant bounded linear operators acting on $H$.
First, we decompose $\mathcal{V}$ into the direct integral of the W*-algebras of bounded operators acting on the reproducing kernel Hilbert spaces $\widehat{H}_\xi$, $\xi\in\widehat{G}$,
generated by the Fourier transform of the reproducing kernel.
Second, we give a constructive criterion for the commutativity of $\mathcal{V}$.
Third, in the commutative case,
we construct a unitary operator
that simultaneously diagonalizes all operators belonging to $\mathcal{V}$,
i.e., converts them into some multiplication operators.
Our scheme generalizes many examples previously studied by Nikolai Vasilevski and other authors.






\end{abstract}

\tableofcontents

\clearpage

\section{Introduction}
\label{sec:intro}

It is well known and easy to see that the radial Toeplitz operators 
on the Bergman space $\cA^2(\bD)$ are diagonal in the monomial basis
and therefore generate a commutative C*-algebra.
In 1999, Vasilevski \cite{Vasilevski1999BergmanToeplitz}
found another non-trivial commutative C*-algebra of operators on the Bergman space.
Namely, he considered ``vertical'' Toeplitz operators, acting in the Bergman space $\cA^2(\Pi)$ over the upper halfplane $\Pi$
and invariant under horizontal translations,
and constructed a unitary operator $R\colon \cA^2(\Pi)\to L^2(\bR_+)$
that simultaneously ``diagonalizes'' all vertical Toeplitz operators, converting them into multiplication operators.
After that, many mathematicians obtained similar results
for other groups of transformations, other spaces of functions, and other domains
\cite{Vasilevski2008book,LoaizaLozano2013,GrudskyKarapetyantsVasilevski2004rad,GrudskyKarapetyantsVasilevski2004par,
GrudskyKarapetyantsVasilevski2004hyp,HutnikHutnikova2011,HutnikovaMiskova2015,RamirezOrtegaSanchezNungaray2015}.
Grudsky, Quiroga, and Vasilevski~\cite{GrudskyQuirogaVasilevski2006} performed a complete study of non-trivial commutative C*-algebras of Toeplitz operators on the weighted Bergman spaces over the unit disk.
Dawson, \'{O}lafsson, and Quiroga-Barranco~\cite{DawsonOlafssonQuiroga2015,DawsonOlafssonQuiroga2018}
showed that in the case of group-invariant operators
acting in the weighted Bergman spaces of analytic functions
over multidimensional domains,
some of the previous results follow naturally from the general theory
of unitary representations of C*-algebras.
Quiroga-Barranco and S\'{a}nchez-Nungaray~\cite{QuirogaBarrancoSanchezNungaray2021} studied commutative C*-algebras of Toeplitz operators in the weighted Bergman spaces over the unit ball using moment maps of the abelian subgroups of the biholomorphism group.

Here we propose another scheme to study group-invariant operators in reproducing kernel Hilbert spaces (RKHS).
We are inspired by the following general idea.
If $G$ is a locally compact group acting on a measure space $D$ such that the translations are unitary operators in $L^2(D)$, and $H$ is a RKHS over $D$ invariant under these translations, then it is natural to expect that the W*-algebra of translation-invariant operators can be described in terms of the Fourier transform (along the orbits of the group action) of the reproducing kernel.
In this paper, we apply this idea to the particular case when $G$ is a locally compact abelian group (LCAG) and the domain $D$ is a ``tube'' $G\times Y$, where $Y$ is just a measure space.
Our scheme is a natural generalization and developement of Vasilevski~\cite{Vasilevski1999BergmanToeplitz},~\cite[Section 3.1]{Vasilevski2008book},
see Example~\ref{example:vertical_analytic}.

\subsection*{Structure of the paper}

In Sections~\ref{sec:Wstar_Stone_Weierstrass} and~\ref{sec:commutativity_of_direct_integral},
we prove two simple general results about W*-algebras: an analog of the Stone--Weierstrass theorem and a criterion of commutativity of a direct integral.

In Section~\ref{sec:invar_L2}, we recall some properties of the Fourier transform and consider 
the unitary representation
of the group $G$ on the space $L^2(G\times Y)$ defined by
\begin{equation}
\label{eq:rho_GY_def}
(\rho_{G\times Y}(a) f)(u,v)\eqdef f(u-a,v)\qquad(a\in G,\ u\in G,\ v\in Y).
\end{equation}
Using the Fourier transform with respect to the first argument,
$F\otimes I_{L^2(Y)}$,
we describe the W*-algebra $\rho_{G\times Y}'$ of bounded linear operators on $L^2(G\times Y)$, commuting with
the horizontal translations $\rho_{G\times Y}(a)$.

The main ideas of Sections~\ref{sec:Wstar_Stone_Weierstrass}--\ref{sec:invar_L2} are well known, but we recall them in a convenient form and state explicitly some results that we have been unable to find in the literature.

In Section~\ref{sec:invar_H}
we suppose that $H$ is a closed subspace of $L^2(G\times Y)$,
invariant under $\rho_{G\times Y}(a)$ for every $a$ in $G$.
Let $\rho_H(a)$ be the compression of $\rho_{G\times Y}(a)$ onto $H$.
Then $\rho_H$ is a unitary representation of $G$ in $H$.
Our principal object of study is the W*-algebra $\cV$
of translation-invariant bounded linear operators acting in $H$,
i.e., the centralizer of the representation $\rho_H$:
\begin{equation}\label{eq:def_V}
\cV \eqdef
\rho_H'
=
\{S\in\cB(H)\colon\quad \forall a\in G\quad
\rho_H(a) S = S \rho_H(a)\}.
\end{equation}
We show that the space $\hH \eqdef (F\otimes I) (H)$ decomposes into the direct integral of some fibers $\hH_\xi\subseteq L^2(Y)$,
and the W*-algebra $\cV$ is spatially isomorphic to the direct integral of the factors $\cB(\hH_\xi)$:
\begin{equation}\label{eq:H_and_V_decomposition}
\hH
=\int_{\Omega}^\oplus\hH_\xi\,\dif\hnu(\xi),\qquad
\Phi \cV\Phi^\ast
=\int_{\Omega}^\oplus\cB(\hH_\xi)\,\dif\hnu(\xi).
\end{equation}

Here $\hG$ is the dual group of $G$, $\hnu$ is the Haar measure on $\hG$ associated with $\nu$, $\Phi\colon H\to\hH$ is the compression of $F\otimes I_{L^2(Y)}$, and $\Omega$ is defined as the set of all ``frequencies'' $\xi$ in $\hG$ corresponding to the non-zero fibers $\hH_\xi$.

In particular, we conclude that $\cV$ is commutative if and only if $\dim \hH_\xi=1$ for $\hnu$-almost all $\xi$ in $\Omega$.
This condition is close to the multiplicity-free condition from \cite{DawsonOlafssonQuiroga2015,DawsonOlafssonQuiroga2018}.

In Section~\ref{sec:invar_RKHS}
we assume that $H$ is a RKHS and denote by $(K_{x,y})_{x\in G,y\in Y}$ the reproducing kernel of $H$.
The translation-invariance of $H$ is equivalent to the following property of the reproducing kernel:
\begin{equation}\label{eq:K_translation_invariance}
K_{x,y}(u,v)
=K_{0,y}(u-x,v)\qquad(x,u\in G,\ y,v\in Y).
\end{equation}
We define $L$ as the Fourier transform of $K$
along the action of the group:
\begin{equation}\label{eq:L_definition}
L_{\xi,y}(v)
\eqdef
(\Phi K_{0,y})(\xi,y)
=
\int_G \overline{\xi(u)}\,
K_{0,y}(u,v)\,\dif\nu(u)
\qquad(\xi\in\hG,\ y,v\in Y).
\end{equation}
Under some technical assumptions, we show that each fiber $\hH_\xi$ is a RKHS, and its reproducing kernel is $(L_{\xi,y})_{y\in Y}$.
As a consequence, we establish a constructive criterion for commutativity of $\cV$, in terms of $L$.

In Section~\ref{sec:commutative_case} we consider the commutative case (when $\dim(\hH_\xi)=1$ for all $\xi$ in $\Omega$) and construct a unitary operator $R$ that simultaneously diagonalizes all operators belonging to $\cV$.
In particular, we diagonalize
Toeplitz operators with translation-invariant generating symbols.

In Section~\ref{sec:non_commutative_case} we consider the non-commutative case with finite-dimensional fibers and construct a unitary operator $R$ that transforms elements of $\cV$ into matrix families.

Finally, in Section~\ref{sec:examples} we apply this scheme to various examples.
One of the examples is new, the others were studied before in many publications.

Our scheme may be viewed as an application of the von Neumann theory to the reproducing kernel Hilbert spaces over domains of the form $G\times Y$.
Here are the main advantages of our approach:
\begin{itemize}
\item we study translation-invariant
operators associated with a general LCAG and acting in a 
general RKHS,
assuming only some technical conditions (possibly, this is the first paper in such settings);
\item we reduce the study of the algebra $\cV$ to the computation of one Fourier integral~\eqref{eq:L_definition};
\item we do not require any analytic or differential structure on the domain.
\end{itemize}
The scheme proposed in this paper unifies many of the currently known results about translation-invariant operators acting in RKHS, but it is not universal.
For example, it cannot be applied to radial operators
on RKHS over the unit ball $\bB_n$ in $\bC^n$ with $n>1$,
because the corresponding unitary group $U(n)$ is not commutative,
and $\bB_n$ does not decompose into a product of the form $U(n)\times Y$.

\section{\texorpdfstring{An analog of the Stone--Weierstrass theorem for subalgebras of $\boldsymbol{L^\infty}$}{An analog of the Stone-Weierstrass theorem for subalgebras of Linfty}}
\label{sec:Wstar_Stone_Weierstrass}

In this section we recall some facts about commutative W*-algebras.
The main result, Theorem~\ref{thm:SW_for_Wstar_alg}, is an analog of the classic Stone--Weierstrass theorem adapted for
W*-subalgebras of $L^\infty(X,\mu)$.
We use some information about W*-algebras from Dixmier~\cite{Dixmier1981vonNeumann},
Sakai~\cite{Sakai1971}, and Takesaki~\cite{Takesaki2002}.

Given a Hilbert space $H$, we denote by $\cB(H)$ the W*-algebra of bounded linear operators acting on $H$
and by $\WOT$ the weak operator topology in $\cB(H)$.
Given a subset $\cS$ of $\cB(H)$, we denote by $\cS'$ the \emph{centralizer} (also called the \emph{commutant}) of $\cS$ in $\cB(H)$.
Given a subset $\cS$ of $\cB(H)$, we denote by $W^\ast(\cS)$ the von Neumann algebra generated by $\cS$.
It is known that $W^\ast(\cS)=(\cS')'$.

In this section, $X$ is a locally compact space and $\mu$ is a Radon measure on $X$.
We suppose that the support of $\mu$ coincides with $X$. Equivalently, $\mu(Y)>0$ for every non-empty open subset $Y$ of $X$.
Then $C_b(X)$ is naturally embedded into $L^\infty(X,\mu)$, where $C_b(X)$ is the C* algebra of bounded continuous functions $X\to\bC$.
For simplicity, we additionally suppose that $X$ is a $\sigma$-compact metric space.
We denote by $\tau_X$ or just by $\tau$
the weak-$\ast$ topology in $L^\infty(X,\mu)$.
Recall that if $(a_n)_{n\in\bN}$ is a bounded sequence in $L^\infty(X,\mu)$
converging pointwise to a function $b$,
then $a_n\xrightarrow{\tau}b$,
i.e., $\int_X a_n f\,\dif\mu\to\int_X b f\,\dif\mu$
for every $f$ in $L^1(X,\mu)$.
Indeed, if $C<+\infty$ and $\|a_n\|_\infty\leq C$ for every $n$,
then the dominated convergence theorem can be applied
with the ``dominant function'' $C|f|$.

Given $b$ in $L^\infty(X,\mu)$, let $M_b\colon L^2(X,\mu)\to L^2(X,\mu)$ be the multiplication operator by $b$:
\[
(M_b f)(x) \eqdef b(x) f(x).
\]
We denote by $\cM_X$ the set of all such multiplication operators:
\[
\cM_X \eqdef \{M_b\colon\ b\in L^\infty(X,\mu)\}.
\]
It is well known and easy to see that $\cM_X$ is a commutative W*-subalgebra of $\cB(L^2(X,\mu))$.
The function $b\mapsto M_b$ is an isometric isomorphism
between the W*-algebras $L^\infty(X,\mu)$ and $\cM_X$.
In particular, $M_{b_1} M_{b_2}=M_{b_1 b_2}=M_{b_2} M_{b_1}$,
$\|M_b\|=\|b\|_\infty$,
and the spectrum of $M_b$ is the essential range of $b$.
The $\tau$-convergence of a net in $L^\infty(X,\mu)$ is equivalent to the $\WOT$-convergence of the corresponding multiplication operators.
It can be shown that
\begin{equation}
\label{eq:multiplication_centralizer}
\cM_X'=\cM_X.
\end{equation}
The following proposition is well known.
It can be proven by applying
Luzin's theorem~\cite[Theorem~7.10]{Folland1999real}
and the Tietze extension theorem,
or by using techniques of C*- and W*-algebras
\cite[proof of Theorem 3.1.2]{Takesaki2002}.

\begin{prop}\label{prop:C_weaklydense_in_Linfty}
Let $Y$ be a compact Hausdorff space
with a Radon measure $\mu_Y$ whose support coincides with $Y$.
Then $\clos_{\tau_Y}(C(Y))=L^\infty(Y,\mu_Y)$.
\end{prop}

The next result is a generalization of Proposition~\ref{prop:C_weaklydense_in_Linfty}
to spaces with infinite measure.
Notice that $\cA$ is not supposed to be closed or dense
in the norm topology of $C_b(X)$.

\begin{thm}\label{thm:SW_for_Wstar_alg}
Let $X$ be a locally compact and $\sigma$-compact metric space,
$\mu$ be a Radon measure on $X$ whose support coincides with $X$, and
$\cA$ be a self-adjoint unital subalgebra of $C_b(X)$
separating points of $X$.
Then $\clos_\tau(\cA)=L^\infty(X,\mu)$.
\end{thm}

\begin{proof}
We denote $\clos_\tau(\cA)$ by $\cW$.
Let $(K_n)_{n\in\bN}$ be an increasing compact covering of $X$.
In the steps 1--4 of the proof,
$Y$ is an arbitrary compact subset of $X$.
We denote by $\mu_Y$ the restriction of $\mu$,
and by $\cA_Y$ and $\cW_Y$ the ``restrictions''
of the algebras $\cA$ and $\cW$, respectively:
\[
\cA_Y\eqdef\{f|_Y\colon\ f\in\cA\},\qquad
\cW_Y\eqdef\{f|_Y\colon\ f\in\cW\}.
\]

Step 1. $\cA_Y$ is a self-adjoint unital subalgebra of $C(Y)$
that separates points of $Y$.
So, by the Stone--Weierstrass theorem,
$\cA_Y$ is dense in $C(Y)$
with respect to the uniform topology.

Step 2. We will prove that $1_Y\in\cW$.
For every $n$ in $\bN$, put
\[
Z_n\eqdef\{x\in K_n\colon\ d(x,Y)\ge 1/n\}.
\]
Using Urysohn's lemma,
choose $f_n\in C(K_n,[0,1])$ such that
\[
f_n(y)=1\quad(y\in Y\cap K_n),\qquad
f_n(x)=0\quad(x\in Z_n).
\]
By Step 1, applied to $K_n$ instead of $Y$, 
we find $g_n$ in $\cA$ such that
$\|g_n|_{K_n}-f_n\|<1/n$.
Put
\[
h_n(x)\eqdef\min\{|g_n(x)|,1\}
=\frac{|g_n(x)|+1-||g_n(x)|-1|}{2}.
\]
Then $h_n$ belong
to the unital C*-algebra generated by $g_n$;
in particular, $h_n\in\cW$.
It is easy to verify that the sequence
$(h_n)_{n\in\bN}$ is bounded in the uniform norm
and converges pointwise to $1_Y$.
Therefore  $h_n\xrightarrow{\tau}1_Y$ and $1_Y\in\cW$.

Step~3. We will prove that $\cW_Y$ is a $\tau_Y$-closed subset of $L^\infty(Y,\mu_Y)$.
Let $(f_j)_{j\in J}$ be a net in $\cW_Y$
that $\tau_Y$-converges to $g\in L^\infty(Y,\mu_Y)$.
Choose $u_j\in\cW$ such that $u_j|_Y=f_j$,
put $v_j=u_j 1_Y$,
and denote by $h$ the extension by zero of the function $g$
to the domain $X$.
Then, by Step~2, $v_j\in\cW$. The assumption $f_j\xrightarrow{\tau_Y}g$ implies that  $v_j\xrightarrow{\tau_X}h$. Therefore $h\in\cW$ and $g=h|_Y\in\cW_Y$.

Step~4. We will prove that $\cW_Y=L^\infty(Y,\mu_Y)$.
Combining Step~1 with Proposition~\ref{prop:C_weaklydense_in_Linfty}
we see that $\clos_{\tau_Y}(\cA_Y)=L^\infty(Y,\mu_Y)$.
Since $\clos_{\tau_Y}(\cA_Y)\subseteq\clos_{\tau_Y}(\cW_Y)=\cW_Y$,
we conclude that $\cW_Y=L^\infty(Y,\mu_Y)$.

Step 5. Let $f\in L^\infty(X,\mu)$.
For every $n$ in $\bN$,
applying the result of Step~4 to the compact $Y=K_n$,
find $g_n$ in $\cW$ such that $g_n|_{K_n}=f|_{K_n}$.
By Step~2, $g_n 1_{K_n}\in\cW$, i.e., $f1_{K_n}\in\cW$.
The sequence $(f 1_{K_n})_{n\in\bN}$ is bounded in the uniform norm and converges pointwise to $f$.
Therefore it converges to $f$ in the topology $\tau_X$,
and $f\in\cW$.
\end{proof}

\section[Criterion for commutativity of a direct integral of W*-algebras]{Criterion for commutativity of a direct integral\\of W*-algebras}
\label{sec:commutativity_of_direct_integral}

For the definition and some properties of direct integrals see, for example,
Dixmier~\cite[Part II, Chapters 1--3]{Dixmier1981vonNeumann},
Folland~\cite[Section~7.4]{Folland2016harmonic}, and 
Takesaki~\cite[Section~4.8]{Takesaki2002}.
In this section,
we assume that 
$(\Omega,\mu)$ is a $\sigma$-finite measure space and
$(H_\xi)_{\xi\in\Omega}$ is a measurable field of non-zero separable Hilbert spaces.
By definition, this concept requires the existence of a ``fundamental sequence of  measurable vector fields'' $(g_j)_{j\in\bN}$ such that for every $\xi$ in $\Omega$, the sequence
$(g_j(\xi))_{j\in\bN}$ is complete in $H_\xi$,
and for every $j,k$ in $\bN$,
the function $\xi\mapsto\langle g_j(\xi),g_k(\xi)\rangle_{H_\xi}$ is measurable.

The following fact about the existence of a ``measurable field of orthonormal bases'' uses the Gram--Schmidt orthogonalization;
see detailed proofs in~\cite[Part~II, Chapter~1, Section~2, Lemma~1]{Dixmier1981vonNeumann},
\cite[Proposition~7.19]{Folland2016harmonic}, and~\cite[Lemma~8.12]{Takesaki2002}.

\begin{prop}
\label{prop:basis_fields}
Let $(H_\xi)_{\xi\in\Omega}$, $(g_j)_{j\in\bN}$
be a measurable field of non-zero separable Hilbert spaces, with dimensions $d_\xi\eqdef\dim(H_\xi)\in\bN\cup\{\infty\}$.
Then $\{\xi\in\Omega\colon\ d_\xi=m\}$ is measurable for every $m$ in $\bN\cup\{\infty\}$.
Moreover, there exists a sequence $(b_j)_{j\in\bN}$
of vector fields with the following properties:
\begin{itemize}
\item for each $\xi\in\Omega$,
$(b_j(\xi))_{j=1}^{d_\xi}$ is an orthonormal basis for $H_\xi$, and $b_j(\xi)=0$ for $j>\dim(H_\xi)$;
\item for each $j$ in $\bN$, there is a measurable partition of $\Omega$, $\Omega=\cup_{k=1}^\infty \Omega_{j,k}$,
such that on each $\Omega_{j,k}$, $b_j(\xi)$ is a finite linear combination of the family $(g_k(\xi))_{k\in\bN}$, with coefficients depending measurably on $\xi$.
\end{itemize}
\end{prop}

We consider the following direct integral of W*-algebras:
\begin{equation}\label{eq:direct_integral}
\cA\eqdef
\int^{\oplus}_\Omega
\cB(H_\xi)\,\dif\mu(\xi).
\end{equation}
Recall that if $S\in\cA$ and
\[
S = \int^{\oplus}_\Omega S(\xi)\,\dif\mu(\xi),
\]
then the norm of $S$ coincides with the essential supremum of the function $\xi\mapsto\|S(\xi)\|_{\cB(H_\xi)}$:
\[
\|S\| \eqdef
\esssup_{\xi,\mu} \|S(\xi)\|_{\cB(H_\xi)}.
\]
In particular, this means that $S=0$ if and only if the equality $S(\xi)=0$ holds for $\mu$-almost all points $\xi$.

\begin{prop}
\label{prop:criterion_commutativity_direct_integral}
$\cA$ is commutative if and only if
$\mu(\Omega_2)=0$, where
\[
\Omega_2 \eqdef
\{\xi\in\Omega\colon \dim(H_\xi)\ge 2\}.
\]
\end{prop}

\begin{proof}
Let $\Omega_1\eqdef\{\xi\in\Omega\colon \dim(H_\xi)=1\}$.
For every $\xi$ in $\Omega_1$,
we have $\dim(H_\xi)=1$,
and $\cB(H_\xi)$ is commutative.

1. Suppose that $\mu(\Omega_2)=0$.
Given $S_1,S_2$ in $\cA$,
the operators $(S_1 S_2)(\xi)$ and $(S_2 S_1)(\xi)$ coincide
for every $\xi$ in $\Omega_1$,
which implies that $S_1 S_2=S_2 S_1$.
So, in this case, $\cA$ is commutative.

2. Suppose that $\mu(\Omega_2)>0$.
We are going to prove that $\cA$ is not commutative.
Let $(b_j)_{j\in\bN}$ be a sequence like in Proposition~\ref{prop:basis_fields}.
In particular, for every $\xi$ in $\Omega_2$, the vectors $b_1(\xi)$ and $b_2(\xi)$ are orthonormal.
Given
\[
f = (f(\xi))_{\xi\in\Omega}
\in\int^{\oplus}_\Omega H_\xi\,\dif\mu(\xi),
\]
we define $S_1 f$ and $S_2 f$ by
\begin{align*}
(S_1 f)(\xi)
&\eqdef
\begin{cases}
\langle f(\xi),b_1(\xi)\rangle b_2(\xi), & \xi\in\Omega_2, \\
0, & \xi\in\Omega_1;
\end{cases}
\\[1ex]
(S_2 f)(\xi)
&\eqdef
\begin{cases}
\langle f(\xi),b_2(\xi)\rangle b_1(\xi), & \xi\in\Omega_2, \\
0, & \xi\in\Omega_1.
\end{cases}
\end{align*}
It is easy to see that $S_1,S_2\in\cA$.
For every $\xi$ in $\Omega_2$,
the restrictions of the operators $S_1(\xi)$ and $S_2(\xi)$ to $\linspan(b_1(\xi),b_2(\xi))$
have the following matrices
with respect to the orthonormal basis $b_1(\xi),b_2(\xi)$:
\[
\begin{bmatrix}
0 & 0 \\ 1 & 0
\end{bmatrix},\qquad
\begin{bmatrix}
0 & 1 \\ 0 & 0
\end{bmatrix}.
\]
In particular,
$\|(S_1 S_2 - S_2 S_1)(\xi)\|_{\cB(H_\xi)} = 1$
for every $\xi$ in $\Omega_2$,
and $S_1 S_2\ne S_2 S_1$.
\end{proof}

\section{\texorpdfstring{Translation-invariant operators in $\boldsymbol{L^2(G\times Y)}$}{Translation-invariant operators in L2(GxY)}}
\label{sec:invar_L2}

Let us recall some well-known concepts and facts related to the translation operators
and to the Fourier transform on LCAG \cite{Folland2016harmonic,HewittRoss1979}.
In this section, we accept the following assumption about $G$ and $Y$.

\begin{assumption}\label{assumption:GY}
Let $G$ be a locally compact abelian group (LCAG) with a Haar measure $\nu$,
and $Y$ be a measure space with a measure $\la$.
We suppose that $G$ is $\sigma$-compact and metrizable, $\la$ is $\sigma$-finite,
and the spaces $L^2(G,\mu)$ and $L^2(Y,\la)$ are separable.
The cartesian product $G\times Y$ is considered with the product measure $\nu\times\la$.
\end{assumption}

We denote by $\hG$ the dual group of $G$.
The conditions on $G$ imply that
$\hG$ is also $\sigma$-compact and metrizable;
see, for example,
\cite[(24.48)]{HewittRoss1979}.
Let $\nu$ be a Haar measure on $G$.
The Fourier transform of a function $f$ in $L^1(G)$ is defined by
\[
(F_1 f)(\xi) \eqdef \int_G \overline{\xi(x)}\,f(x)\,\dif\nu(x)\qquad(\xi\in\hG).
\]
Let $\hnu$ be the dual Haar measure on $\hG$,
such that $\|F_1 f\|_{L^2(\hG,\hnu)}=\|f\|_{L^2(G)}$
for every $f$ in $L^1(G)\cap L^2(G)$.
We write $L^p(\hG)$ instead of $L^p(\hG,\hnu)$
and denote by $F$ the Fourier--Plancherel transform which coincides with $F_1$ on $L^1(G)\cap L^2(G)$

Given $a$ in $G$, we denote by $\rho_G(a)$
the translation operator acting in $L^2(G)$ by the rule
\[
(\rho_G(a) f)(x)\eqdef f(x-a).
\]
Given $a$ in $G$, we denote by $E_a$
the function $\hG\to\bC$ defined by $E_a(\xi)\eqdef\xi(a)$.
Let $\rho_{\hG}(a)$ be the operator of multiplication by $E_{-a}$:
\begin{equation}\label{eq:rho_hG_def}
\rho_{\hG}(a)\eqdef M_{E_{-a}}.
\end{equation}
It is well known and easy to see that
$(\rho_G,L^2(G))$ and $(\rho_{\hG},L^2(\hG))$
are (strongly continuous) unitary representations of $G$,
and the Fourier--Plancherel transform intertwines them:
\begin{equation}\label{eq:shift_and_mul_by_character}
F \rho_G(a) F^\ast = \rho_{\hG}(a).
\end{equation}
We shortly denote by $\rho_{\hG}'$
the centralizer of the set $\{\rho_{\hG}(a)\colon\ a\in G\}$.
A similar notation is used through the paper also for other unitary representations.

The following proposition describes
the operators acting in $\cB(L^2(\hG))$
and commuting with the multiplications by characters of $\hG$.

\begin{prop}\label{prop:commuting_with_mul_by_characters}
$W^\ast(\{E_a\colon a\in G\})=L^\infty(\hG)$, and
$\rho_{\hG}'=\cM_{\hG}$.
\end{prop}

\begin{proof}
The first statement follows from Theorem~\ref{thm:SW_for_Wstar_alg}
and the fact that the set $\{E_{-a}\colon\ a\in G\}$
separates the points of $\hG$
(see, for example, \cite[Theorem (22.17)]{HewittRoss1979}).
The second statement is a consequence of formula~\eqref{eq:multiplication_centralizer}.
\end{proof}

An operator $A$ of the class $\cB(L^2(G))$
is called a \emph{multiplier} of $L^2(G)$
if $A$ commutes with $\rho_G(a)$ for every $a$ in $G$.
The next proposition,
being an equivalent form
of Proposition~\ref{prop:commuting_with_mul_by_characters},
means that the Fourier--Plancherel transform converts every multiplier of $L^2(G)$ into a multiplication operator in $L^2(\hG)$.
See Larsen~\cite[proof of Theorem~4.1.1]{Larsen1971}
for a more constructive proof.

\begin{prop}\label{prop:multipliers}
$F \rho_G' F^\ast = \cM_{\hG}$.
\end{prop}

Here $F \rho_G' F^\ast$ is a short notation for
$\{FAF^\ast\in\cB(L^2(G))\colon\ \forall a\in G\quad
\rho_G(a) A = A \rho_G(a)\}$.

\begin{cor}\label{cor:commuting_with_mul_by_characters}
Let $\Omega$ be a measurable subset of $\hG$
and let $A\in\cB(L^2(\Omega))$.
Suppose that $A$ commutes with the multiplications
by all characters of $\hG$ restricted to $\Omega$:
\[
\forall a\in G\qquad A M_{E_a|_\Omega} = M_{E_a|_\Omega} A.
\]
Then $A\in\cM_{\Omega}$,
i.e., there exists $b$ in $L^\infty(\Omega)$ such that $A=M_b$.
\end{cor}

\begin{proof}
Define $B\colon L^2(\hG)\to L^2(\hG)$ by the following rule:
\[
(B f)(\xi) \eqdef
\begin{cases}
(A f|_\Omega)(\xi), & \xi\in\Omega;\\
0, & \xi\notin\Omega.
\end{cases}
\]
It is easy to see that $B M_{E_a}=M_{E_a} B$ for every $a$ in $G$.
By Proposition~\ref{prop:commuting_with_mul_by_characters},
there exists $b_1\in L^\infty(\hG)$ such that $B=M_{b_1}$.
Put $b=b_1|_\Omega$.
Then $A=M_b$.
\end{proof}

Now we pass to the domains $G\times Y$ and $\hG\times Y$, the spaces $L^2(G\times Y)$ and $L^2(\hG\times Y)$, and the natural unitary representations of $G$ in these spaces.
It is well known~\cite[Part~II, Chapter~1, Section~8, Proposition~11 and its Corollary]{Dixmier1981vonNeumann} that
\begin{equation}\label{eq:L2_product}
L^2(\hG\times Y)
=L^2(\hG) \otimes L^2(Y)
=\int_{\hG}^{\oplus}L^2(Y)\,\dif\hnu(\xi).
\end{equation}
Put $\rho_{G\times Y}(a)\eqdef \rho_{G}(a)\otimes I_{L^2(Y)}$ for each $a\in G$. More explicitly, $\rho_{G\times Y}$ is defined by~\eqref{eq:rho_GY_def}.
Then, $\rho_{G\times Y}$
is a unitary representation of $G$ in $L^2(G\times Y)$.
We are going to understand the structure of the centralizer $\rho_{G\times Y}'$.
The crucial role here is played by the operator $F\otimes I_{L^2(Y)}$ which we abbreviate as $F\otimes I$ and call ``the Fourier transform with respect to the first coordinate''.

For each $a$ in $G$, we put $\rho_{\hG\times Y}(a) \eqdef \rho_{\hG}(a)\otimes I_{L^2(Y)}$,
i.e.,
\begin{equation}
\label{eq:rho_hG_Y_explicit}
(\rho_{\hG\times Y}(a) g)(\xi, y)
=E_{-a}(\xi) g(\xi,y)
\qquad(a\in G,\ \xi\in\hG,\ y\in Y).
\end{equation}
Then $\rho_{\hG\times Y}$ is a unitary representation of $G$ in $L^2(\hG\times Y)$.
Formula~\eqref{eq:shift_and_mul_by_character} implies that $F\otimes I$ intertwines $\rho_{G\times Y}$ with $\rho_{\hG\times Y}$:
\begin{equation}
\label{eq:from_rho_GY_to_rho_hGY}
(F\otimes I)\rho_{G\times Y}(a)(F\otimes I)^\ast
=\rho_{\hG\times Y}(a).
\end{equation}

\begin{lem}
\label{lem:weak_convergence_and_tensor_products}
Let $H_1$ and $H_2$ be separable Hilbert spaces,
$(A_j)_{j\in J}$ be a net in $\cB(H_1)$,
and $B\in\cB(H_1)$.
Then $(A_j\otimes I_{H_2})_{j\in J}$ weakly converges to $B\otimes I_{H_2}$ if and only if $(A_j)_{j\in J}$ weakly converges to $B$.
\end{lem}

\begin{proof}
Given $f,g$ in $H_1$ and $u,v$ in $H_2$,
\begin{align*}
\langle (A_j\otimes I_{H_2})f\otimes u, g\otimes v\rangle_{H_1\otimes H_2}
&= \langle A_j f, g\rangle_{H_1}\,
\langle u,v\rangle_{H_2},
\\
\langle
(B\otimes I)f\otimes u,
g\otimes v
\rangle_{H_1\otimes H_2}
&=\langle Bf,g\rangle_{H_1}\,\langle u,v
\rangle_{H_2}.
\end{align*}
These identities yield immediately the sufficiency part.
For the necessity part, we take $u$ and $v$ to be the same normalized vector in $H_2$.
\end{proof}

\begin{lem}
\label{lem:tensor_product_by_identity}
Let $H_1$ and $H_2$ be separable Hilbert spaces, and $S$ be a selfadjoint subset of $\cB(H_1)$. Then
\[
W^\ast(\{A\otimes I_{H_2}\colon A\in S\})
= W^\ast(S)\otimes
(\bC I_{H_2}).
\]
\end{lem}

\begin{proof}
Let $R$ be the unital algebra generated by $S$,
and
$P=\{A\otimes I_{H_2}\colon A\in R\}$.
Then, obviously, $P$ is the unital algebra generated by
$\{A\otimes I_{H_2}\colon A\in S\}$.
Furthermore,
\begin{align*}
W^\ast(\{A\otimes I_{H_2}\colon A\in S\})
&=\clos_{\WOT}(P)
=\{B\otimes I_{H_2}\colon B\in\clos_{\WOT}(R)\}
\\
&=\clos_{\WOT}(R)
\otimes (\bC I_{H_2})
=W^\ast(S)
\otimes(\bC I_{H_2}).
\end{align*}
The second equality in this chain follows from Lemma~\ref{lem:weak_convergence_and_tensor_products}.
\end{proof}

\begin{prop}
\label{prop:rhoGY_centralizer}
\begin{equation}\label{eq:rhoGY_commutant}
(F\otimes I) \rho_{G\times Y}^{\prime} (F\otimes I)^\ast
=\int_{\hG}^{\oplus}\cB(L^2(Y))\,\dif\hnu.
\end{equation}
Equivalently,
\begin{equation}\label{eq:rhoGY_commutant2}
(F\otimes I) \rho_{G\times Y}^{\prime} (F\otimes I)^\ast
=\cM_{\hG}\otimes \cB(L^2(Y)).
\end{equation}
\end{prop}

\begin{proof}
Since $F\otimes I$ is a unitary operator and
$(F\otimes I)\rho_{G\times Y}(F\otimes I)^\ast=\rho_{\hG\times Y}$,
we have
\[
(F\otimes I)(\rho_{G\times Y})' (F\otimes I)^* = (\rho_{\hG\times Y})'.
\]
Furthermore, by Lemma~\ref{lem:tensor_product_by_identity},
\[
W^*(\rho_{\hG\times Y})
=W^*(\{\rho_{\hG}(a)\otimes I_{L^2(Y)}\colon a\in G\})
=W^*(\rho_{\hG})\otimes
(\bC I_{L^2(Y)}).
\]
Now we apply the fact \cite[Theorem 2.8.1]{Sakai1971}, \cite[Theorem 5.9]{Takesaki2002} that the centralizer of the tensorial product
is the tensor product of the corresponding centralizers,
and use
Proposition~\ref{prop:commuting_with_mul_by_characters}:
\[
\rho_{\hG\times Y}'
=W^*(\rho_{\hG\times Y})'
= W^*(\rho_{\hG})'\otimes \cB(L^2(Y))
= \cM_{\hG}\otimes \cB(L^2(Y)).
\]
We have proven~\eqref{eq:rhoGY_commutant2}.
Furthermore, it is well known
(see a more general result in \cite[Corollary~8.30]{Takesaki2002}) that
\[
\cM_{\hG}=\int_{\hG}^{\oplus}\bC\,\dif\hnu.
\]
Now, using the ``distributive relation'' between the direct integral and the tensor product of von Neumann algebras
\cite[Part~II, Chapter~3, Section~4, Proposition~4]{Dixmier1981vonNeumann}, we obtain~\eqref{eq:rhoGY_commutant2}:
\[
\rho_{\hG\times Y}'
= \cM_{\hG}\otimes \cB(L^2(Y))
= \left(\int_{\hG}\bC\,\dif\hnu\right)\otimes \cB(L^2(Y))
= \int_{\hG} \cB(L^2(Y))\,\dif\hnu.
\qedhere
\]
\end{proof}

The next corollary gives a constructive recipe for the decomposition~\eqref{eq:rhoGY_commutant}.

\begin{cor}
\label{cor:Axi_explicit}
Let $S\in\rho_{G\times Y}'$.
For every $\xi$ in $\hG$, define
$A_\xi\colon L^2(Y)\to L^2(Y)$ by
\begin{equation}\label{eq:Axi_recipe}
(A_\xi h)(v) = \frac{(F\otimes I) S(f\otimes h)(\xi,v)}{(F_1 f)(\xi)}\qquad(h\in L^2(Y)),
\end{equation}
where $f$ is any function of the class $L^1(G)\cap L^2(G)$ such that its Fourier transform $F_1 f$ does not vanish, and $(f\otimes h)(u,v)\eqdef f(u)h(v)$. 
Then
\begin{equation}\label{eq:S_explicit_decomposition}
(F\otimes I) S (F\otimes I)^\ast
= \int^\oplus_{\hG}
A_\xi\,\dif\hnu(\xi).
\end{equation}
\end{cor}

\begin{proof}
The existence of a family $(A_\xi)_{\xi\in\hG}$  satisfying~\eqref{eq:S_explicit_decomposition}
follows from Proposition~\ref{prop:rhoGY_centralizer}.
We are going to prove~\eqref{eq:Axi_recipe}.
Let $f\in L^1(G)\cap L^2(G)$ such that $F_1 f$ does not vanish,
and let $h\in L^2(Y)$.
Put $g\eqdef F_1 f=F f$.
Then $g\otimes h=(F\otimes I)(f\otimes h)$, and
\begin{align*}
((F\otimes I)S(f\otimes h))(\xi,v)
&=
((F\otimes I)S(F\otimes I)^\ast)(g\otimes h))(\xi,v)
\\
&= 
(A_\xi (g\otimes h)(\xi,\cdot))(v)
=
(A_\xi (g(\xi) h))(v)
=
(F_1 f)(\xi)\cdot (A_\xi h)(v).
\end{align*}
Dividing by $(F_1 f)(\xi)$ we get~\eqref{eq:Axi_recipe}.
\end{proof}

Corollary~\ref{cor:Axi_explicit} (and thereby Proposition~\ref{prop:rhoGY_centralizer}) can be proved with a more direct and elementary reasoning,
similarly to Larsen~\cite[proof of Theorem~4.1.1]{Larsen1971}.

\section{Translation-invariant operators in Hilbert spaces}
\label{sec:invar_H}

In this section, we make the following assumption.

\begin{assumption}\label{assumption:H}
Additionally to Assumption~\ref{assumption:GY},
let $H$ be a closed subspace of $L^2(G\times Y)$,
and $P\colon L^2(G\times Y)\to L^2(G\times Y)$ be the orthogonal projection
with $P(L^2(G\times Y))=H$.
We suppose that $H$ is an invariant subspace of 
the representation $\rho_{G\times Y}$.
Equivalently, $P$ commutes with $\rho_{G\times Y}(a)$ for all $a$ in $G$.
\end{assumption}

Recall that the unitary representation $\rho_H$ and its centralizer $\cV\eqdef\rho_H'$ were defined in Section~\ref{sec:intro}, see~\eqref{eq:def_V}.
Using the general tools from previous sections, in this section we easily obtain a decomposition of $\cV$.

Let $\hH\eqdef(F\otimes I)(H)$
and let $\hP$ be the orthogonal projection acting in $L^2(\hG\times Y)$ such that $\hP(L^2(\hG\times Y))=\hH$.
Equivalently,
\[
\hP
= (F\otimes I) P (F\otimes I)^\ast.
\]

\begin{prop}\label{prop:P_hH_decomposition}
There exists a family of orthogonal projections
$(\hP_\xi)_{\xi\in\hG}$ acting in $L^2(Y)$
such that
\begin{equation}\label{eq:P_decomposition}
\hP
=\int^{\oplus}_{\hG}\,\hP_\xi\,\dif\hnu(\xi).
\end{equation}
\end{prop}

\begin{proof}
Since $P\in\rho_{G\times Y}'$,
by Proposition~\ref{prop:rhoGY_centralizer}, there exists a family $(\hP_\xi)_{\xi\in\hG}$ in $L^2(Y)$ such that $(F\otimes I) P (F\otimes I)^\ast$ decomposes into the direct integral~\eqref{eq:P_decomposition}.
We have that $\hP^2=\hP$ and $\hP^\ast=\hP$.
By well-known properties of the direct integral~\cite[formula~(7.24)]{Folland2016harmonic},
\[
\int^{\oplus}_{\hG}
\hP_\xi^2\,\dif\hnu(\xi)
=\int^{\oplus}_{\hG}
\hP_\xi\,\dif\hnu(\xi),\qquad
\int^{\oplus}_{\hG}
\hP_\xi^\ast\,\dif\hnu(\xi)
=\int^{\oplus}_{\hG}
\hP_\xi\,\dif\hnu(\xi).
\]
Therefore, the equalities $\hP_\xi^2=\hP_\xi$ and $\hP_\xi^\ast=\hP_\xi$
are fulfilled for almost every $\xi$ in $\hG$.
After modifying $\hP_\xi$ on a set of zero measure, we assure these properties for all $\xi$ in $\hG$.
\end{proof}

\begin{rem}
\label{rem:P_xi_explicit}
Formula~\eqref{eq:Axi_recipe} yields an explicit expression for $\hP_\xi$:
\begin{equation}\label{eq:hP_xi_explicit}
(\hP_\xi h)(v)
= \frac{((F\otimes I) P(f\otimes h))(\xi,v)}{(F_1 f)(\xi)}\qquad(h\in L^2(Y)),
\end{equation}
where $f$ is any function of the class $L^1(G,\nu)\cap L^2(G,\nu)$ such that its Fourier transform $F_1 f$ does not vanish.
\end{rem}

In the rest of this section, we fix a family $(\hP_\xi)_{\xi\in\hG}$ as in Proposition~\ref{prop:P_hH_decomposition}.
For each $\xi$ in $\hG$, we denote by $\hH_\xi$ the image of the operator $\hP_\xi$
and by $d_\xi$ its dimension:
\begin{equation}\label{eq:hH_xi_def}
\hH_\xi\eqdef\hP_\xi(L^2(Y)),\qquad
d_\xi\eqdef\dim(\hH_\xi).
\end{equation}
Furthermore, we 
denote by $\Omega$ the set of the frequencies corresponding to the non-trivial fibers:
\begin{equation}\label{eq:def_Omega}
\Omega\eqdef\{\xi\in\hG\colon\
d_\xi>0\}.
\end{equation}

\begin{prop}
\label{prop:Hxi_measurable}
$(\hH_\xi)_{\xi\in\Omega}$ is a measurable field of Hilbert spaces.
Moreover, there exists a sequence of measurable vector fields $(q_j)_{j\in\bN}$ with the following properties:
\begin{itemize}
\item[(i)] $(q_{j,\xi})_{j=1}^{d_\xi}$ is an orthonormal basis for $\hH_\xi$, and $q_{j,\xi}=0$ for $j>\dim(\hH_\xi)$,
\item[(ii)] for each $j$ in $\bN$, the function $\Omega\times Y\to\bC$, $(\xi,v)\mapsto q_{j,\xi}(v)$, is measurable.
\end{itemize}
\end{prop}

\begin{proof}
Given an orthonormal basis $(e_j)_{j\in\bN}$ in $L^2(Y)$, we put
\[
g_{j,\xi}
\eqdef \hP_\xi e_j.
\]
Then $(g_{j,\xi})_{j\in\bN}$ is complete in $\hH_\xi$ for each $\xi$.
Due to~\eqref{eq:hP_xi_explicit}, the functions $(\xi,v)\mapsto g_{j,\xi}(v)$ are measurable on $\Omega\times Y$.

Applying Proposition~\ref{prop:basis_fields} we get a family $(q_{j,\xi})_{j\in\bN,\xi\in\Omega}$ with desired properties.
Indeed, if $\Omega_{j,k}$ are as Proposition~\ref{prop:basis_fields},
then $(\xi,v)\mapsto q_{j,\xi}(v)$ is measurable on $A_{j,k}\times Y$ being a finite linear combination of measurable functions $(\xi,v)\mapsto g_{k,\xi}(v)$.

We notice that the measurability in this sense (as functions defined on $\Omega\times Y$) is stronger then the measurability which appears in the definition of a measurable field of Hilbert spaces.
\end{proof}

\begin{prop}
\label{prop:hH_is_direct_integral}
$\hH$ is the direct integral of the spaces $\hH_\xi$:
\label{prop:hH_decomposition}
\begin{equation}
\label{eq:hH_decomposition}
\hH = \int^{\oplus}_{\Omega}\,\hH_\xi\,\dif\hnu(\xi).
\end{equation}
\end{prop}

\begin{proof}
If $g\in\hH$ and $g_\xi\eqdef g(\xi,\cdot)$ for every $\xi$,
then $\hP_\xi g_\xi=g_\xi$ for almost every $\xi$.
After modifying $g$ on a set of measure zero, if needed, we assume that $\hP_\xi g_\xi = g_\xi$ for all $\xi$ in $\Omega$
and $g_\xi=0$ for every $\xi$ in $\hG\setminus\Omega$.
So, the family $(g_\xi)_{\xi\in\Omega}$ belongs to the direct integral in the right-hand side of~\eqref{eq:hH_decomposition}.

Conversely, given a vector field $(g_\xi)_{\xi\in\Omega}$ belonging to the right-hand side of~\eqref{eq:hH_decomposition},
we trivially extend $g_\xi=0$ for $\xi$ in $\hG\setminus\Omega$ and obtain a function $g$ of the class $L^2(\hG\times Y)$ such that $\hP g=g$.
\end{proof}

Let $\Phi\colon H\to\hH$ be defined by $\Phi(f)\eqdef(F\otimes I)(f)$.
In other words, $\Phi$ is the compression of $F\otimes I$ to the domain $H$ and codomain $\hH$.

\begin{thm}
\label{thm:general_V_decomposition}
With Assumption~\ref{assumption:H},
\begin{equation}\label{eq:general_V_decomposition}
\Phi \cV \Phi^\ast
= \int^{\oplus}_\Omega\,
\cB(\hH_\xi)\,\dif\hnu(\xi).
\end{equation}
\end{thm}

\begin{proof}
We will explain the inclusion $\subseteq$ only.
Let $S\in\cV$.
Define $A\in\cB(L^2(G\times Y))$ by
$Af\eqdef SPf$.
Since $S$ takes values in $H$,
we obtain $PA=PAP=AP$.
Furthermore, Assumption~\ref{assumption:H} implies that $P\in\rho_{G\times Y}'$
and therefore $A\in\rho_{G\times Y}'$.
By Proposition~\ref{prop:rhoGY_centralizer},
there exists a family $(B_\xi)_{\xi\in\hG}$ in $\cB(L^2(Y))$ such that
\[
(F\otimes I) A (F\otimes I)^\ast = \int_{\hG}^{\oplus}B_\xi\,\dif\hnu(\xi).
\]
Since $A$ commutes with $P$,
we conclude that $(F\otimes I) A (F\otimes I)^\ast$ commutes with $\hP$.
By~\eqref{eq:P_decomposition},
for almost all $\xi$ we obtain that $B_\xi$ commutes with $\hP_\xi$,
i.e., $\hH_\xi$ is an invariant subspace of $B_\xi$.
Let $D_\xi$ be the compression of $B_\xi$ to $\hH_\xi$.
Then for every $g$ in $\hH$ and almost every $\xi$ in $\Omega$,
\[
(\Phi S \Phi^\ast g)
(\xi,\cdot)
=((F\otimes I) A (F\otimes I)^\ast g)
(\xi,\cdot)
= B_\xi g(\xi,\cdot)
= D_\xi g(\xi,\cdot).
\]
For $\xi$ in $\hG\setminus\Omega$, the space $\hH_\xi$ is trivial, and we omit these values of $\xi$.
So,
\[
\Phi S \Phi^\ast
=\int_\Omega D_\xi\,\dif\hnu(\xi).
\qedhere
\]
\end{proof}

\begin{prop}
\label{prop:commutativity_criterion_H}
$\cV$ is commutative
if and only if $d_\xi=1$
for $\hnu$-almost every point $\xi$ of $\Omega$.
\end{prop}

\begin{proof}
Follows from Proposition~\ref{prop:criterion_commutativity_direct_integral}
and Theorem~\ref{thm:general_V_decomposition}.
\end{proof}

\section{Translation-invariant operators in RKHS}
\label{sec:invar_RKHS}

In this section, we consider the case when $H$ is a RKHS over $G\times Y$.
We freely use some basic properties of RKHS.
See, for example,
Aronszajn~\cite{Aronszajn1950}
or Agler and McCarthy~\cite{AglerMcCarthy2002}.

First, we give a simple criterion for $\rho_{G\times Y}$-invariance of $H$ in terms of the reproducing kernel.
This is a particular case of
\cite[Proposition~4.1]{MaximenkoTelleria2020}.

\begin{prop}\label{prop:K_shift}
Let $G$ and $Y$ satisfy Assumption~\ref{assumption:GY},
and let $H$ be a RKHS over $G\times Y$,
with reproducing kernel $(K_{x,y})_{(x,y)\in G\times Y}$.
Then the following conditions are equivalent.
\begin{enumerate}[label=(\alph*)]
\item $\rho_{G\times Y}(H)\subseteq H$ for every $a$ in $G$.
\item $P \rho_{G\times Y}(a)=\rho_{G\times Y}(a) P$
for every $a$ in $G$,
where $P$ is the orthogonal projection on $L^2(G\times Y)$ such that $P(L^2(G\times Y))=H$.
\item For every $x,u$ in $G$ and every $y,v$ in $Y$,
\begin{equation}
\label{eq:K_shift}
K_{x,y}(u,v)=K_{0,y}(u-x,v).
\end{equation}
\item For every $a,x$ in $G$ and every $y$ in $Y$,
\begin{equation}
\label{eq:translate_K}
\rho_{G\times Y}(a) K_{x,y}=K_{a+x,y}.
\end{equation}
\end{enumerate}
\end{prop}

In the rest of this section, we make the following assumption.

\begin{assumption}\label{assumption:RKHS}
Additionally to~Assumption~\ref{assumption:H},
suppose that $H$ is a RKHS over $G\times Y$,
and the reproducing kernel
$(K_{x,y})_{(x,y)\in G\times Y}$ satisfies
\begin{equation}\label{eq:K_integral_bounded}
\forall y\in Y\qquad
\sup_{v\in Y} \int_G |K_{0,y}(u,v)|\,\dif\nu(u)<+\infty.
\end{equation}
\end{assumption}

For every $\xi$ in $\hG$ and every $y,v$ in $Y$, we define $L_{\xi,y}(v)$ by~\eqref{eq:L_definition}.
In particular, \eqref{eq:K_integral_bounded} implies that the integral in \eqref{eq:L_definition} exists in the Lebesgue sense,
and for every $y,v$ in $Y$ the function $\xi\mapsto L_{\xi,y}(v)$ is continuous.

The goal of this section is to provide more constructive descriptions of the projections $\hP_\xi$ and spaces $\hH_\xi$ than in Section~\ref{sec:invar_H}.

Using Proposition~\ref{prop:K_shift}
and the Hermitian property of $K$
we can write $P$ as
\begin{equation}\label{eq:P_as_convolution}
(P f)(x,y)= \int_Y \int_G f(u,v) K_{0,v}(x-u,y)\,\dif\nu(u)\,\dif\la(v).
\end{equation}
The inner integral in the right-hand side of~\eqref{eq:P_as_convolution} is a convolution.
The following lemma can be viewed
as an application of the convolution theorem to this inner integral.
The technical assumptions on $(G,\mu)$, $(Y,\la)$, and $K$ allow us to interchange the order of integration.

\begin{lem}\label{lem:hP_via_L_for_L1_functions}
Let $f\in L^1(G\times Y)\cap L^2(G\times Y)$. Then for every $\xi$ in $\hG$ and every $y$ in $Y$,
\begin{equation}\label{eq:Phi_P_via_L}
((F\otimes I) P f)(\xi,y)
= \int_Y
((F\otimes I) f)(\xi,v)\,
\overline{L_{\xi,y}(v)}\,\dif\la(v).
\end{equation}
Equivalently,
\begin{equation}\label{eq:Phi_P_via_L_inner_product}
((F\otimes I) P f)(\xi,y)
= \langle ((F\otimes I) f)(\xi,\cdot), L_{\xi,y} \rangle_{L^2(Y)}.
\end{equation}
\end{lem}

\begin{proof}
Step 1.
We denote by $C_y$ the supremum in \eqref{eq:K_integral_bounded}.
Let us estimate from above the following triple integral:
\[
J \eqdef \int_Y \int_G \int_G
|f(u,v)|\,|K_{x,y}(u,v)|\,\dif\nu(x)\,\dif\nu(u)\,\dif\la(v).
\]
We write $K_{x,y}(u,v)$ as $K_{0,y}(u-x,y)$,
make the change of variables $t=u-x$ (where $u$ is a fixed parameter),
apply Tonelli's theorem and assumption~\eqref{eq:K_integral_bounded}:
\begin{align*}
J
&= \int_Y \int_G \int_G
|f(u,v)|\,|K_{0,y}(t,v)|\,\dif\nu(t)\,\dif\nu(u)\,\dif\la(v)
\\
&= \int_G \int_Y
|f(u,v)|\,\left(\int_G |K_{0,y}(t,v)|\,\dif\nu(t)\right)\dif\la(v)\,\dif\nu(u)
\\
&\le C_y \int_G \int_Y
|f(u,v)|\,\dif\la(v)\,\dif\nu(u)
= C_y \|f\|_{L^1(G\times Y)}
< +\infty.
\end{align*}
Step 2. Due to Step~1, we can apply Fubini's theorem to the following integrals.
\begin{align*}
((F\otimes I) P f)(\xi,y)
&=\int_G \int_G \int_Y
\overline{\xi(x)} f(u,v) \overline{K_{x,y}(u,v)}\,\dif\la(v)\,\dif\nu(u)\,\dif\nu(x)
\\
&=\int_Y \int_G \overline{\xi(u)} f(u,v)
\overline{
\left(\,\int_G \overline{\xi(u-x)} K_{0,v}(u-x,y)\,\dif\nu(x)\right)}\dif\nu(u)\,\dif\la(v)
\\
&=
\int_Y ((F\otimes I) f)(\xi,v) \overline{L_{\xi,y}(v)}\,\dif\la(v).
\qedhere
\end{align*}
\end{proof}

\begin{lem}
\label{lem:L_repr_property}
For every $y,v$ in $Y$ and every $\xi$ in $\hG$,
\begin{equation}\label{eq:L_repr_property}
L_{\xi,y}(v)=\langle L_{\xi,y},L_{\xi,v}\rangle_{L^2(Y)}.
\end{equation}
\end{lem}

\begin{proof}
Follows from Lemma~\ref{lem:hP_via_L_for_L1_functions} applied to $f=K_{0,y}$.
\end{proof}

\medskip
The following general fact can be seen as a corollary from Moore--Aronszajn theorem. We have not found the explicit statement of this fact in the bibliography.
In many applications, $\cH_1$ is a space of square-integrable \emph{functions}, rather than their equivalence classes.

\begin{prop}[about RKHS generated by a reproducing family in a complete space with pre-inner product]
\label{prop:RKHS_embedded_into_HS}
Let $X$ is a set and $\cH_1$ be a space of functions $X\to\bC$ with a pre-inner product
$\langle\cdot,\cdot\rangle_{\cH_1}$,
not necessarily strictly positive.
We suppose that $\cH_1$ is complete with respect to $\langle\cdot,\cdot\rangle_{\cH_1}$.
Let $(\cK_x)_{x\in X}$ be a family in $\cH_1$ such that
\begin{equation}\label{eq:cK_reproducing_property}
\forall x,y\in X\qquad \cK_x(y)
= \langle \cK_x, \cK_y \rangle_{\cH_1}.
\end{equation}
Let
\[
\cH_2 \eqdef \{f\in\cH_1\colon\quad
\forall x\in X\quad f(x)=\langle f,\cK_x\rangle\}.
\]
Then $\cH_2$ is a RKHS
and $(\cK_x)_{x\in X}$ is the reproducing kernel of $\cH_2$.
The rule
\begin{equation}\label{eq:cP_def}
(\cP f)(x) \eqdef \langle f,\cK_x\rangle_{\cH_1},
\end{equation}
defines an orthogonal projection in $\cH_1$,
and $\cP(\cH_1)=\cH_2$.
\end{prop}

\begin{proof}
The main challenge is to prove that $\cP f\in\cH_1$ for every $f$ in $\cH_1$.
We will get this fact indirectly, using the existence of an orthogonal projection onto a closed subspace of a Hilbert space.
Since $\langle\cdot,\cdot\rangle_{\cH_1}$ is not necessarily strictly positive, we have to pass from elements of $\cH_1$ to equivalent classes and return back.

Let $\cH_0\eqdef\{f\in\cH_1\colon\ \langle f,f\rangle_{\cH_1}=0\}$.
Then $\cH_0$ is a closed subspace of $\cH_1$
and $\cH_1/\cH_0$ is a Hilbert space
(with a strictly positive inner product).
We denote by $\pi_1$ the canonical projection $\cH_1\to\cH_1/\cH_0$.

Condition~\eqref{eq:cK_reproducing_property} easily implies that $(\cK_x)_{x\in X}$ is a positive definite kernel.
Let $\cH_3$ be the span of $\{\cK_x\colon\ x\in X\}$ and $\cH_4$ be the RKHS constructed in the Moore--Aronszajn theorem.
Due to~\eqref{eq:cK_reproducing_property}, the inner product in $\cH_4$ is inherited from $\cH_1$.
The elements of $\cH_4$ are pointwise limits of Cauchy sequences in $\cH_3$.
At this point, we know that $\cH_4\subseteq\cH_2$.

Since $\pi_1(\cH_4)$ is a closed subset of $\cH_1/\cH_0$, there exists an orthogonal projection $\cP_1$ in $\cH_1/\cH_0$ such that
$\cP_1(\cH_1/\cH_0)=\pi_1(\cH_4)$.
Given $f$ in $\cH_1$, let $g\in\cH_4$ be such a function that $\cP_1(\pi_1(f))=\pi_1(g)$.
Then, for every $x$ in $X$,
\begin{equation}\label{eq:cP_via_cP1}
\begin{aligned}
g(x)
&=\langle g,\cK_x\rangle_{\cH_1}
=\langle \pi_1(g),\pi_1(\cK_x)\rangle_{\cH_1/\cH_0}
=\langle \cP_1(\pi_1(f)),\pi_1(\cK_x)\rangle_{\cH_1/\cH_0}
\\
&=\langle \pi_1(f),\cP_1(\pi_1(\cK_x))\rangle_{\cH_1/\cH_0}
=\langle \pi_1(f),\pi_1(\cK_x)\rangle_{\cH_1/\cH_0}
=\langle f,\cK_x\rangle_{\cH_1}
=(\cP f)(x).
\end{aligned}
\end{equation}
Thereby we get $\cP f=g\in\cH_4$.
So, $\cP$ is a well-defined function $\cH_1\to\cH_1$.
Computation~\eqref{eq:cP_via_cP1} means that
$\pi_1\circ\cP=\cP_1\circ\pi_1$.
Since $\cP_1$ is a bounded selfadjoint linear operator
and $\pi_1$ is a linear isometry,
we easily conclude that $\cP$ is a bounded autoadjoint linear operator.

If $f\in\cH_2$ and $g\in\cH_4$ such that
$\pi_1(g)=\cP_1(\pi_1(f))$, then the definition of $\cH_2$ and the reproducing property in $\cH_4$ imply that $f=g$.
Hence, $\cH_4=\cH_2$.
Finally, we can conclude that $\cP(\cH_1)=\cH_2$
and $\cP^2=\cP$.
\end{proof}

For every $\xi$ in $\hG$, we define $\hP_\xi\colon L^2(Y)\to L^2(Y)$ by
\begin{equation}\label{eq:hPxi_def}
(\hP_\xi h)(y)
\eqdef \langle h,L_{\xi,y}\rangle_{L^2(Y)}
=
\int_Y h(v) \overline{L_{\xi,y}(v)}\,\dif\la(v).
\end{equation}
Then, we denote by~$\hH_\xi$ the image of $\hP_\xi$:
\begin{equation}
\label{eq:hHxi_def}
\hH_\xi\eqdef \hP_\xi(L^2(Y)).
\end{equation}
We will to prove that \eqref{eq:hPxi_def} is equivalent to the definition of $\hP_\xi$ in Section~\ref{sec:invar_H}.

\begin{thm}
\label{thm:hH_xi_via_L}
Let Assumption~\ref{assumption:RKHS} hold.
For every $\xi$ in $\hG$,
$\hP_\xi$ is an orthogonal projection in $L^2(Y)$
and $\hH_\xi$ is a RKHS with reproducing kernel $(L_{\xi,y})_{y\in Y}$.
For each $\xi$ in $\hG$,
\begin{equation}\label{eq:hH_xi_via_comb_L}
\hH_\xi
=
\clos_{L^2(Y)}(\linspan(\{L_{\xi,y}\colon y\in Y\})),
\end{equation}
where the closure is understood as the set of the pointwise limits of Cauchy sequences.
Moreover,
\begin{equation}\label{eq:hP_as_direct_integral}
\hP = \int_{\hG} \hP_\xi\,\dif\hnu(\xi),
\end{equation}
i.e., for every $g$ in $L^2(\hG\times Y)$,
\begin{equation}\label{eq:hP_via_L}
(\hP g)(\xi,y)
=
\langle g(\xi,\cdot), L_{\xi,y}\rangle_{L^2(Y)}
=
\int_Y g(\xi,v)\,\overline{L_{\xi,y}(v)}\,\dif\la(v)\qquad(\xi\in\hG,\ y\in Y).
\end{equation}
\end{thm}

\begin{proof}
The first statements follow from 
Lemma~\ref{lem:L_repr_property} and Proposition~\ref{prop:RKHS_embedded_into_HS}.
Let $A$ be the operator in $L^2(\hG\times Y)$ defined by the the right-hand side of~\eqref{eq:hP_as_direct_integral} or~\eqref{eq:hP_via_L}.
For each $\xi\in\hG$,
$\|\hP_\xi\|\le 1$.
This easily implies that $A$ is a bounded linear operator with $\|A\|\le 1$.

By Lemma~\ref{lem:hP_via_L_for_L1_functions}, the equality $(F\otimes I) P f=A (F\otimes I) f$
holds for every $f$ in the intersection $L^2(G\times Y)\cap L^1(G\times Y)$, which is a dense subset of $L^2(G\times Y)$.
Since $(F\otimes I) P$ and $A (F\otimes I)$ are bounded linear operators,
we conclude that the equality
$(F\otimes I) P=A (F\otimes I)$
holds on the whole space $L^2(G\times Y)$.
Hence,
$\hP=(F\otimes I) P(F\otimes I)^\ast=A$.
\end{proof}

\begin{cor}\label{cor:repr_property_hH}
Let $g\in\hH$. Then for almost all $\xi$ in $\hG$ and almost all $y$ in $Y$,
\begin{equation}\label{eq:repr_property_hH}
g(\xi,y)
= \langle g(\xi,\cdot), L_{\xi,y}\rangle_{L^2(Y)}
= \int_Y g(\xi,v) \overline{L_{\xi,y}(v)}\,\dif\la(v).
\end{equation}
\end{cor}

\begin{rem}
The integral in \eqref{eq:repr_property_hH} is taken over $Y$,
not over the $\hG\times Y$.
In general, $\hH$ does not have to be a RKHS.
\end{rem}

\begin{prop}\label{prop:dim_hH_eq_int_L}
For every $\xi$ in $\hG$,
\begin{equation}\label{eq:dim_hH_eq_int_L}
\dim(\hH_\xi)
=\int_Y L_{\xi,y}(y)\,\dif\la(y)
=\int_Y \|L_{\xi,y}\|_{L^2(Y)}^2\,\dif\la(y).
\end{equation}
\end{prop}

\begin{proof}
This is a general formula for the dimension of the image of the orthogonal projection defined as an integral operator.
Let us outline the proof in our settings.
Recall that 
$(q_{j,\xi})_{j=1}^{d_\xi}$ is an orthonormal basis for $\hH_\xi$.
Therefore,
$L_{\xi,y}(y)=\sum_{j=1}^{d_\xi}|q_{j,\xi}(y)|^2$ and
\[
\int_Y L_{\xi,y}(y)\,\dif\la(y)
=\sum_{j=1}^{d_\xi}
\int_Y |q_{j,\xi}(y)|^2\,\dif\la(y)
=\sum_{j=1}^{d_\xi} 1
=d_\xi.
\qedhere
\]
\end{proof}

As a consequence of Proposition~\ref{prop:dim_hH_eq_int_L}, we get a constructive description of $\Omega$:
\begin{equation}
\label{eq:Omega_via_int_L}
\Omega = \left\{\xi\in\hG\colon\quad
\int_Y L_{\xi,y}(y)\,\dif\la(y)>0\right\}.
\end{equation}

\begin{thm}\label{thm:commutativity_criterion_in_RKHS}
With Assumption~\ref{assumption:RKHS}, the following conditions are equivalent.
\begin{itemize}
\item[(a)] $\cV$ is commutative.
\item[(b)] For every $\xi$ in $\Omega$,\quad $\dim(\hH_\xi)=1$.
\item[(c)] For every $\xi$ in $\Omega$,
\begin{equation}\label{eq:int_L_le_1}
\int_Y L_{\xi,y}(y)\,\dif\la(y)=1.
\end{equation}
\item[(d)] For every $\xi$ in $\Omega$ and every $y,v$ in $Y$,
\begin{equation}\label{eq:Schwarz_equality}
|L_{\xi,y}(v)|^2
=L_{\xi,y}(y)L_{\xi,v}(v).
\end{equation}
\item[(e)] There exists a family
$(q_\xi)_{\xi\in\Omega}$ in $L^2(Y)$ such that
the function $(\xi,v)\mapsto q_\xi(v)$ is measurable,
the function $q_\xi$ forms
an orthonormal basis of $\hH_\xi$, and
\begin{equation}
\label{eq:L_factorization}
L_{\xi,y}(v)=\overline{q_\xi(y)}q_\xi(v)\qquad(\xi\in\Omega,\ y,v\in Y).
\end{equation}
\end{itemize}
\end{thm}

\begin{proof}
The major part of the proof follows from Propositions~\ref{prop:commutativity_criterion_H} and~\ref{prop:dim_hH_eq_int_L}.
We will comment only a few missing ideas.
If \eqref{eq:int_L_le_1} holds for almost every $\xi$, then, by continuity of $L_{\xi,y}(v)$ with respect to $\xi$, it holds for every $\xi$.

Condition (d) means that the Schwarz inequality for $L_{\xi,y}$ and $L_{\xi,v}$ reduces to an equality, i.e.,
the functions $L_{\xi,y}$ and $L_{\xi,v}$ are linear dependent.
Since $y$ and $v$ are arbitrary elements of $Y$
and $\{L_{\xi,y}\colon y\in Y\}$ is a total subset of $\hH_\xi$,
(d) implies (b).

If (b) holds, then we apply Proposition~\ref{prop:Hxi_measurable} with $d_\xi=1$
and obtain a family $(q_\xi)_{\xi\in\Omega}$
such that $(\xi,v)\mapsto q_\xi(v)$ is measurable
and $q_\xi$ is an orthonormal basis of $\hH_\xi$.
The reproducing kernel of $\hH_\xi$ expresses through this orthonormal basis by~\eqref{eq:L_factorization}.
\end{proof}

\begin{rem}
\label{rem:q_via_L}
In the context of the last part of the proof,
for every $\xi$ in $\Omega$,
there exists $z$ in $Y$ and $\tau$ in $\bC$ (both depending on $\xi$) such that
$\|L_{\xi,z}\|\ne0$, $|\tau|=1$, and
\[
q_\xi=\tau\, \frac{L_{\xi,z}}{\|L_{\xi,z}\|}.
\]
This means that $q$ is essentially determined by $L$.
In many examples, a decomposition of the form~\eqref{eq:L_factorization} with a measurable function $q$ is obvious.
\end{rem}

\begin{rem}\label{rem:differences_in_RKHS}
Let us emphasize additional properties that obtain $\hP_\xi$ and $\hH_\xi$ when passing from Assumption~\ref{assumption:H} to Assumption~\ref{assumption:RKHS}.
\begin{enumerate}
\item Now $\hP_\xi$ and $\hH_\xi$ are uniquely defined \emph{for every} $\xi$, instead of almost everywhere.
\item $\hP_\xi$ and $\hH_\xi$ have
\emph{simple explicit expressions}
in terms of $(L_{\xi,y})_{y\in Y}$.
\item The elements of $\hH_\xi$, in contrast to $L^2(Y)$, can be treaten as \emph{functions}, instead of classes of equivalence.
\item We have \emph{direct formulas}~\eqref{eq:dim_hH_eq_int_L} and~\eqref{eq:Omega_via_int_L}
to compute $\Omega$ and
the dimensions of $\hH_\xi$.
\item Theorem~\ref{thm:commutativity_criterion_in_RKHS} is a \emph{constructive criterion} for the commutativity of $\cV$.
\end{enumerate}
\end{rem}

\section{Diagonalization in the commutative case}
\label{sec:commutative_case}

In this section we assume that Assumption~\ref{assumption:H} is fulfilled and, additionally, $d_\xi=\dim(\hH_\xi)=1$ for every $\xi$ in $\Omega$.
In this case, Proposition~\ref{prop:hH_decomposition} implies that there exists a family of functions $(q_\xi)_{\xi\in\Omega}$
with the following properties:
\begin{itemize}
\item[(i)] $\hH_\xi=\bC q_\xi$ and $\|q_\xi\|_{L^2(Y)}=1$
for every $\xi$ in $\Omega$;
\item[(ii)] the function $\Omega\times Y\to\bC$, $(\xi,v)\mapsto q_\xi(v)$,
is measurable.
\end{itemize}
For each $\xi$ in $\Omega$,
the function $q_\xi$ is uniquely defined, up to a constant of absolute value $1$.

In particular, if $H$ is a RKHS satisfying Assumption~\ref{assumption:RKHS}
and equivalent conditions from Theorem~\ref{thm:commutativity_criterion_in_RKHS}, then $q_\xi$ is usually easy to find from $L_\xi$,
see Remark~\ref{rem:q_via_L}.

Identifying $\hH_\xi$ and $\cB(\hH_\xi)$ with $\bC$,
in this section we will simplify the descomposition from Theorem~\ref{thm:general_V_decomposition}
and construct a unitary operator $R\colon H\to L^2(\Omega)$ such that $R \cV R^\ast=\cM_{\Omega}$.
Our treatment generalizes ideas from Vasilevski~\cite{Vasilevski2008book}.

Define $N\colon\hH\to L^2(\Omega)$ by
\begin{equation}\label{eq:N_def}
(Ng)(\xi)
\eqdef \langle g(\xi,\cdot),q_\xi\rangle_{L^2(Y)}
=\int_Y \overline{q_\xi(v)}\,g(\xi,v)\,\dif\la(v).
\end{equation}

\begin{prop}\label{prop:N_properties}
$N$ is a unitary operator, and its inverse $N^\ast\colon L^2(\Omega)\to\hH$ acts by the following rule:
\begin{equation}
\label{eq:N_adjoint}
(N^\ast h)(\xi,y)=
\begin{cases}
q_\xi(y) h(\xi), & \xi \in \Omega; \\
0, & \xi\in\hG\setminus\Omega.
\end{cases}
\end{equation}
\end{prop}

\begin{proof}
1. Let $g\in\hH$.
For every $\xi$ in $\Omega$,
by Proposition~\ref{prop:hH_is_direct_integral} we have $g(\xi,\cdot)\in\hH_\xi$.
Since $\hH_\xi=\bC q_\xi$ and
$\|q_\xi\|_{L^2(Y)}^2=1$,
we obtain
$\|g(\xi,\cdot)\|_{L^2(Y)}=|(Ng)(\xi)|$.
Hence, $N$ is isometric:
\[
\|Ng\|_{L^2(\Omega)}^2
=\int_{\Omega}|(Ng)(\xi)|^2\,\dif\hnu(\xi)
=\int_{\Omega}\|g(\xi,\cdot)\|_{L^2(Y)}^2\,
\dif\hnu(\xi)
=\|g\|_{\hH}^2.
\]
2. Let $Z$ be the operator defined by the right-hand side of \eqref{eq:N_adjoint}.
Proposition~\ref{prop:hH_is_direct_integral} assures that $Zh$ indeed belongs to $\hH$ and $Z$ is well-defined.
A simple direct computation yields $NZh=h$, which completes the proof.
\end{proof}

We define $R\colon H\to L^2(\Omega)$ by the following rule:
\begin{equation}\label{eq:R_def}
R \eqdef N \Phi,
\end{equation}
i.e.,
\begin{equation}
\label{eq:R_direct_definition}
(R f)(\xi)
= \int_Y ((F\otimes I) f)(\xi,v) \overline{q_\xi(v)}\,\dif\la(v).
\end{equation}

\begin{rem}\label{rem:R_similar_to_audio_codecs}
The idea of the operator $R$ is similar to the ideas of some lossless audio- and video-codecs: it is a kind of a Fourier transform followed by a ``general compression''.
\end{rem}

\begin{prop}
\label{prop:R_is_unitary}
$R$ is a unitary operator from $H$ onto $L^2(\Omega)$.
\end{prop}

\begin{proof}
Indeed, $R$ is the composition of two unitary operators.
\end{proof}

\begin{prop}
\label{prop:RK}
Let $y\in Y$ and $\xi\in\Omega$.
Then
\begin{equation}\label{eq:RK}
(RK_{0,y})(\xi)
=\conju{q_\xi(y)}.
\end{equation}
\end{prop}

\begin{proof}
$(RK_{0,y})(\xi)
=
\langle
(\Phi K_{0,y})(\xi,\cdot),q_\xi
\rangle_{L^2(Y)}
=
\langle
L_{\xi,y},q_\xi\rangle_{L^2(Y)}
= \conju{q_\xi(y)}$.
\end{proof}

\begin{rem}
Additionally to the operators $N\colon\hH\to L^2(\Omega)$ and $R\colon H\to L^2(\Omega)$, one can define in a similar way their extended versions
$\tN\colon L^2(\hG\times Y)\to L^2(\Omega)$
and $\tR\colon L^2(G\times Y)\to L^2(\Omega)$.
Then
\[
\tN^\ast \tN=\hP,\qquad
\tN \tN^\ast=I_{L^2(\Omega)},\qquad
\tN^\ast(L^2(\Omega))=\hH,
\]
\[
\tR^\ast \tR = P,\qquad
\tR \tR^\ast
= I_{L^2(\Omega)},\qquad
\tR^\ast(L^2(\Omega))=H.
\]
\end{rem}

\medskip
We recall that $E_a$ is defined by $E_a(\xi)=\xi(a)$,
where $a\in G$ and $\xi\in\hG$.

\begin{prop}\label{prop:R_rhoGY_Rast}
Let $a\in G$. Then
\begin{equation}\label{eq:R_rhoGY_Rast}
R \rho_H(a) R^\ast
= M_{E_{-a}|_\Omega}.
\end{equation}
\end{prop}

\begin{proof}
Let $h\in L^2(\Omega)$.
Substituting the definitions and using~\eqref{eq:from_rho_GY_to_rho_hGY} we easily get
\[
R \rho_H(a) R^\ast h
=N \Phi \rho_H(a) \Phi^\ast N^\ast h
=N (F\otimes I) \rho_{G\times Y}(a) (F\otimes I)^\ast N^\ast h
=
N (\rho_{\hG}(a)\otimes I) N^\ast h.
\]
Therefore, for every $\xi$ in $\Omega$,
\[
(R \rho_H(a) R^\ast h)(\xi)
= \langle E_{-a}(\xi) q_\xi h(\xi), q_\xi \rangle_{L^2(Y)}
= E_{-a}(\xi) h(\xi).
\qedhere
\]
\end{proof}

\begin{thm}\label{thm:R_diagonalizes_invariant_operators}
Define $\Lambda\colon L^\infty(\Omega)\to\cV$ by
$\Lambda(\sigma)\eqdef R^\ast M_\sigma R$.
Then $\Lambda$ is an isometric isomorphism of W*-algebras.
In particular, for every $S$ in $\cV$, the product $R S R^\ast$
is a multiplication operator in $L^2(\Omega)$.
\end{thm}

\begin{proof}
The algebraic properties of $\Lambda$ and the isometric property of $\Lambda$
follow easily from well-known properties of multiplication operators
and from the fact that $R$ is a unitary operator.

We have to show that $\Lambda$ is surjective.
Let $S\in\cV$
and $B\eqdef R S R^\ast$.
By Proposition~\ref{prop:R_rhoGY_Rast},
for each $a$ in $G$ we have
\[
B M_{E_a|_\Omega}
= (R S R^\ast)(R\rho(-a) R^\ast)
= R S \rho_H(-a) R^\ast
= R \rho_H(-a) S R^\ast
= M_{E_a|_\Omega} B,
\]
i.e., $B$ commutes with $E_a$.
By Corollary~\ref{cor:commuting_with_mul_by_characters},
we conclude that $B\in\cM_\Omega$.
\end{proof}

In particular, Theorem~\ref{thm:R_diagonalizes_invariant_operators} means that the W*-algebras $\cV$ and $\cM_\Omega$ are spatially isomorphic.
Figure~\ref{fig:main_diagram} shows a commutative diagram corresponding to the formula $S=\Lambda(\sigma)=R^\ast M_\sigma R$ from  Theorem~\ref{thm:R_diagonalizes_invariant_operators}, jointly with some auxiliary objects.

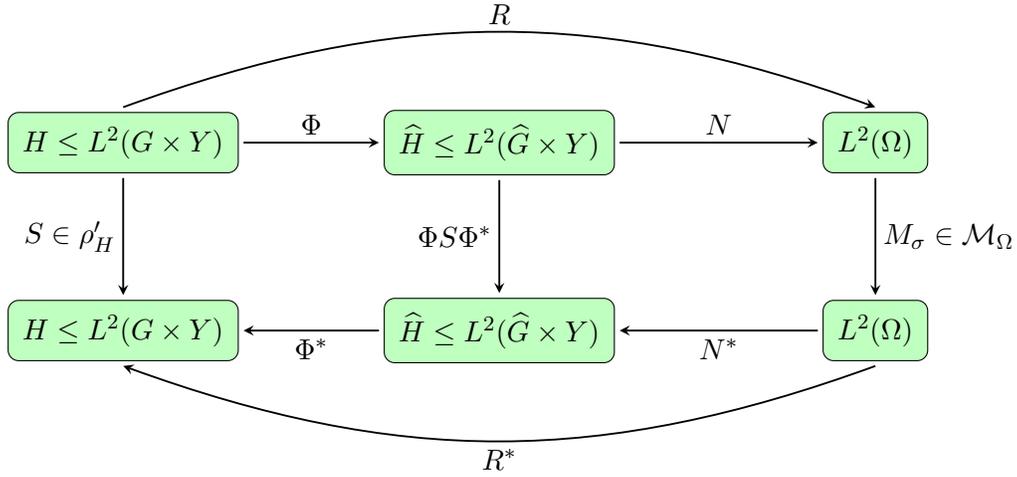
\begin{figure}[ht]
\centering
\begin{tikzpicture}
  [outer sep=2pt,inner sep=1pt,ar/.style={line width=0.7pt,-stealth},
  bignode/.style={draw,
  rectangle,rounded corners,
  inner sep=0.5em,
  fill=green!50!white!50}]
\node [bignode] (H) at (0,0)
  {$H\le L^2(G\times Y)$};
\node [bignode] (hH) at (5,0) {$\hH\le L^2(\hG\times Y)$};
\node [bignode] (L2O) at (10,0) {$L^2(\Omega)$};
\node [bignode] (dH) at (0,-2.5) {$H\le L^2(G\times Y)$};
\node [bignode] (dhH) at (5,-2.5) {$\hH\le L^2(\hG\times Y)$};
\node [bignode] (dL2O) at (10,-2.5) {$L^2(\Omega)$};
\draw[ar] (H)--(hH) node[midway,above] {$\Phi$};
\draw[ar] (hH)--(L2O) node[midway,above] {$N$};
\draw[ar] (dhH)--(dH)
  node[midway,below] {$\Phi^\ast$};
\draw[ar] (dL2O)--(dhH)
  node[midway,below] {$N^\ast$};
\draw[ar,bend left=20]
  (H.north) to
  node [midway,above] {$R$}
  (L2O.north);
\draw[ar,bend left=20]
  (dL2O.south) to
  node [midway,below] {$R^\ast$}
  (dH.south);
\draw[ar] (H)--(dH)
  node[midway,left] {$S\in\rho_H'$};
\draw[ar] (hH)--(dhH)
  node [midway,left] {$\Phi S\Phi^\ast$};
\draw[ar] (L2O)--(dL2O)
  node[midway,right] {$M_\sigma\in\cM_\Omega$};
\end{tikzpicture}
\caption{Operators participating in Theorem~\ref{thm:R_diagonalizes_invariant_operators}.
\label{fig:main_diagram}}
\end{figure}

Given $S$ in $\cV$,
we say that $\sigma\eqdef\Lambda^{-1}(S)$ is the \emph{spectral function} of the operator $S$.
The next corollary provides is an explicit formula for $\sigma$.

\begin{cor}
\label{cor:Lambda_inv_explicit}
Let $S\in\cV$.
Then for every $\xi$ in $\Omega$,
\begin{equation}
\label{eq:Lambda_ast_explicit}
(\Lambda^{-1}(S))(\xi)
=\frac{(R S K_{0,y})(\xi)}%
{\conju{q_\xi(y)}},
\end{equation}
where $y$ is an arbitrary point of $Y$ such that $q_\xi(y)\ne0$.
\end{cor}

\begin{proof}
Let $\sigma=\Lambda^{-1}(S)$, i.e., $RS=M_\sigma R$.
Furthermore, let $\xi\in\Omega$ and $y\in Y$ such that $q_\xi(y)\ne0$.
Using~\eqref{eq:RK} we obtain
\[
(R S K_{0,y})(\xi)
= (M_\sigma R K_{0,y})(\xi)
= \sigma(\xi) \conju{q_\xi(y)}.
\]
Dividing by $\conju{q_\xi(y)}$ we get~\eqref{eq:Lambda_ast_explicit}.
\end{proof}

\begin{cor}\label{cor:spectrum_via_sigma}
Let $S\in\cV$ and $\sigma=\Lambda^{-1}(S)$.
Then $\|S\|=\|\sigma\|_\infty$, and the spectrum of $S$ is the essential range of $\sigma$.
\end{cor}

\subsection*{Berezin transform of a translation-invariant operator in terms of its spectral function}

\begin{prop}
\label{prop:Ber_invar_oper}
Let Assumption~\ref{assumption:RKHS} holds,
$S\in\cV$, and $\sigma=\Lambda^{-1}(S)$.
Then
\begin{equation}\label{eq:Ber_invar_oper}
\Ber(S)(x,y)=\frac{\int_\Omega \sigma(\xi) |q_\xi(y)|^2\,\dif\hnu(\xi)}{\int_\Omega |q_\xi(y)|^2\,\dif\hnu(\xi)}
\qquad(x\in G,\ y\in Y).
\end{equation}
In particular, $\Ber(S)(x,y)$ does not depend on $x$.
\end{prop}

\begin{proof}
Recall that the Berezin transform $\Ber(S)$ of $S$ is defined by
\[
\Ber(S)(x,y)
\eqdef
\frac{\langle S K_{x,y},K_{x,y}\rangle}%
{\langle K_{x,y},K_{x,y}\rangle}\qquad(x\in G,\ y\in Y).
\]
Now we apply~\eqref{eq:translate_K} and the hypothesis that $S$ commutes with $\rho_H(x)$:
\begin{align*}
\Ber(S)(x,y)
&=\frac{1}{\|K_{x,y}\|^2} \langle S K_{x,y}, K_{x,y} \rangle
=\frac{1}{\|\rho_{H}(x)K_{0,y}\|^2} \langle S \rho_{H}(x) K_{0,y}, \rho_{H}(x) K_{0,y} \rangle
\\
&=\frac{1}{\|K_{0,y}\|^2} \langle \rho_{H}(x) S K_{0,y}, \rho_{H}(x)K_{0,y} \rangle
=\frac{1}{\|K_{0,y}\|^2} \langle S K_{0,y}, K_{0,y} \rangle
\\
&=\frac{1}{\|K_{0,y}\|^2} \langle R^\ast M_\sigma R K_{0,y},K_{0,y}\rangle
=\frac{1}{\|K_{0,y}\|^2} \langle M_\sigma R K_{0,y},R K_{0,y}\rangle.
\end{align*}
Substituting~\eqref{eq:RK} we get~\eqref{eq:Ber_invar_oper}.
\end{proof}

\subsection*{Spectral functions of Toeplitz operators with translation-invariant generating symbols}

Given $\phi\in L^\infty(G\times Y)$, we denote by $T_\phi$ the \emph{Toeplitz operator with generating symbol $\phi$}, acting in $H$ by
\[
T_\phi(f)\eqdef P(\phi f) = P M_\phi f.
\]
In the following proposition we compute the spectral function of $T_\phi$, supposing that $\phi$ depends only on the $Y$-component.

\begin{prop}\label{prop:gamma_Toeplitz}
Let $\psi\in L^\infty(Y)$.
Define $\phi\in L^\infty(G\times Y)$ by $\phi(u,v)\eqdef\psi(v)$.
Then $T_\phi\in\cV$ and $T_\phi=\Lambda(\gamma_\psi)$,
where
$\gamma_\psi\colon\Omega\to\bC$ is defined by
\begin{equation}\label{eq:gamma_Toeplitz_def}
\gamma_\psi(\xi) \eqdef \int_Y \psi(v) |q_\xi(v)|^2\,\dif\la(v).
\end{equation}
\end{prop}

\begin{proof}[First proof]
It is easy to see that $T_\phi$ commutes with the horizontal translations $\rho_H(a)$, $a\in G$.
Since $\phi(u,v)$ does not depend on $u$,
the operator $M_\phi$ commutes with $F\otimes I$, and
\begin{equation}
\label{eq:FI_P_Mphi}
(F\otimes I) P M_\phi
= \hP (F\otimes I) M_\phi
= \hP M_\phi (F\otimes I).
\end{equation}
Let $\xi\in\Omega$ and $y\in Y$ such that $q_\xi(y)\ne0$.
Using~\eqref{eq:FI_P_Mphi}
we simplify $RSK_{0,y}$:
\begin{align*}
(R S K_{0,y})(\xi)
&= (N (F\otimes I) P M_\phi K_{0,y})(\xi)
= (N \hP M_\phi (F\otimes I) K_{0,y})(\xi)
= \langle \hP_\xi M_\phi L_{\xi,y}, q_\xi\rangle_{L^2(Y)}
\\
&=
\langle M_\phi L_{\xi,y}, \hP_\xi q_\xi\rangle_{L^2(Y)}
=
\langle M_\phi \conju{q_{\xi}(y)} q_\xi, q_\xi\rangle_{L^2(Y)}
=
\conju{q_{\xi}(y)}
\int_Y \psi(v) |q_\xi(v)|^2\,\dif\la(v).
\end{align*}
With the help of~\eqref{eq:Lambda_ast_explicit} we conclude that $\Lambda^{-1}(T_\phi)=\gamma_\psi$.
\end{proof}

\begin{proof}[Second proof]
Let us verify directly that
$R T_\phi R^\ast = M_{\gamma_\psi}$.
Given $h$ in $L^2(\Omega)$, we simplify $R T_\phi R^\ast h$ applying~\eqref{eq:FI_P_Mphi}:
\begin{equation}\label{eq:R_T_R_ast_auxiliar_step}
R T_\phi R^\ast h
=N (F\otimes I)P M_\phi (F\otimes I)^\ast N^\ast h
=N \hP M_\phi N^\ast h.
\end{equation}
If $\xi\in\Omega$ and $v$ in $Y$, then
$(M_\phi N^\ast h)(\xi,v)
=h(\xi) \psi(v) q_\xi(v)$.
Therefore,
\[
(\hP M_\phi N^\ast h)(\xi,v)
= h(\xi) (\hP_\xi(\psi\,q_\xi))(v),
\]
and
\begin{align*}
(R T_\phi R^\ast h)(\xi)
&= \langle (\hP M_\phi N^\ast h)(\xi,\cdot), q_\xi\rangle_{L^2(Y)}
= h(\xi)\,
\langle \hP_\xi(\psi q_\xi), q_\xi\rangle_{L^2(Y)}
\\
&=  h(\xi)\,
\langle \psi q_\xi, q_\xi\rangle_{L^2(Y)}
= h(\xi) \gamma_\psi(\xi).
\qedhere
\end{align*}
\end{proof}

We denote by $\cVT_0$ the set of all Toeplitz operators of the form $T_\phi$,
where $\phi$ is as in Proposition~\ref{prop:gamma_Toeplitz},
and by $\cG_0$ the set of the spectral functions of such Toeplitz operators:
\begin{equation}\label{eq:cG0_def}
\cG_0\eqdef\{\gamma_\psi\colon\ \psi\in L^\infty(Y)\}.
\end{equation}
Let $\cVT$ and $\cG$ be
the C*-algebras generated by $\cVT_0$ and $\cG_0$, respectively.

\begin{cor}\label{cor:cVT_and_cG}
The C*-algebra $\cVT$ is the image of the C*-algebra $\cG$
with respect to the isometric isomorphism $\Lambda$.
The C*-algebra $\cVT$ is weakly dense in $\cV$
if and only if the C*-algebra $\cG$ is dense in $L^\infty(\Omega)$
with respect to the weak-* topology $\tau_\Omega$.
\end{cor}

\begin{proof}
$\Lambda$ is an isometrical isomorphism $L^\infty(\Omega)\to\cV$,
and is restriction to $\cG$ is an isometrical isomorphism from $\cG$ onto $\cVT$.
Moreover, $\Lambda$ maps the weak-* topology of $L^\infty(\Omega)$
onto the weak operator topology in $\cV$.
Therefore, $\cVT$ is weakly dense in $\cV$ if and only if
$\cG$ is dense in $(L^\infty(\Omega),\tau_\Omega)$.
\end{proof}

Corollary~\ref{cor:cVT_and_cG} provides us with a tool to study the C*-algebra $\cVT$ generated by Toeplitz operators with translation-invariant generating symbols.
A natural problem is to find the C*-algebra generated by all Toeplitz operators with bounded symbols (not necesarily translation-invariant),
acting in a RKHS $H$.
Various characterizations of this Toeplitz algebra have been found for the Bergman and Segal--Bargmann--Fock spaces,
see Xia~\cite{Xia2015},
Bauer and Fulsche~\cite{BauerFulshe2020}, and Hagger~\cite{Hagger2020}.
Much earlier, Engli\v{s}~\cite{Englis1992} proved that Toeplitz operators
acting in the Bergman space $L^2_{\text{hol}}(\bD)$
are weakly dense in $\cB(L^2_{\text{hol}}(\bD))$.

\section{Non-commutative case with finite-dimensional fibers}
\label{sec:non_commutative_case}

This section is a generalization of the previous one.
In this section we require Assumption~\ref{assumption:H}
and additionally suppose that $d_\xi\eqdef\dim(\hH_\xi)$ is finite for every $\xi$ in $\Omega$.
Let $(q_{j,\xi})_{j\in\bN,\xi\in\Omega}$ be a measurable basis family for the spaces $\hH_\xi$, like in Proposition~\ref{prop:Hxi_measurable}.
For each $\xi$ in $\Omega$, we denote by $Q_\xi$ the column-vector-function
\[
Q_\xi(v) \eqdef
\bigl[ q_{j,\xi}(v) \bigr]_{j=1}^{d_\xi}\qquad(v\in Y).
\]
Its conjugate transpose is the row-vector-function
\[
Q_\xi^\ast(v)
=
\left(\bigl[\,\conju{q_{j,\xi}(v)}\,\bigr]_{j=1}^{d_\xi}\right)^\top
=
\bigl[\,\conju{q_{1,\xi}(v)},\ldots,\conju{q_{d_\xi,\xi}(v)}\, \bigr]_{j=1}^{d_\xi}.
\]
Since $\hH_\xi$ is finite-dimensional, it is a RKHS over $Y$, and its reproducing kernel $(L_{\xi,y})_{y\in Y}$ can be expressed via the orthonormal basis $q_{1,\xi},\ldots,q_{d_\xi,\xi}$ of $\hH_\xi$:
\begin{equation}\label{eq:L_via_xis}
L_{\xi,y}(v)
=\sum_{j=1}^{d_\xi}
\overline{q_{j,\xi}(y)}\,
q_{j,\xi}(v)
=Q_\xi^\ast(y) Q_\xi(v).
\end{equation}
When Assumption~\ref{assumption:RKHS} holds, $L$ can be computed in terms of $K$ by~\eqref{eq:L_definition}, and in some examples one can find functions $q_{j,\xi}$ decomposing $L$ like in~\eqref{eq:L_via_xis}.

This section has many similarities with the previous one, thereby we omit detailed proofs.

We denote by $\cX$ the following direct integral of Hilbert spaces $\bC^{d_\xi}$:
\[
\cX \eqdef \int^{\oplus}_\Omega \bC^{d_\xi}\,\dif\hnu(\xi).
\]
The elements of $\cX$ are classes of equivalence of  vector sequences,
component-wise measurable on $\{\xi\in\Omega\colon\ d_\xi=m\}$ for every $m$,
and square-integrable. 
We define $N\colon\hH\to\cX$ by
\begin{equation}\label{eq:N_multidimensional}
(N g)(\xi)
\eqdef
\left[ \langle g(\xi,\cdot), q_{j,\xi}\rangle_{\hH_\xi}\right]_{j=1}^{d_\xi}
=\left[ \int_Y \overline{q_{j,\xi}(v)} g(\xi,v)\,\dif\la(v)\right]_{j=1}^{d_\xi}
=\int_Y \conju{Q_\xi(v)}\,g(\xi,v)\,\dif\la(v).
\end{equation}

\begin{prop}\label{prop:N_multidimensional_properties}
$N$ is a unitary operator from $\hH$ onto $\cX$.
Its inverse $N^\ast\colon\cX\to\hH$
acts by the following rule:
\begin{equation}\label{eq:Nast_multidimensional}
(N^\ast h)(\xi,y)
=\sum_{j=1}^{d_\xi}q_{j,\xi}(y)h_j(\xi)
=Q_\xi^\top(y) h(\xi)
\qquad(h\in\cX,\ \xi\in\hG,\ y\in Y).
\end{equation}
\end{prop}

\begin{proof}
We apply Proposition~\ref{prop:hH_is_direct_integral} and use the isomorphism between $\hH_\xi$ and $\bC^\xi$
induced by the orthonormal basis $(q_{j,\xi})_{j=1}^{d_\xi}$.
\end{proof}

In particular, formula~\eqref{eq:Nast_multidimensional} tells us that
$(N^\ast h)(\xi,y)=0$ for $\xi$ in $\hG\setminus\Omega$,
because the corresponding sum in~\eqref{eq:Nast_multidimensional} is empty.

We define $R\colon H\to\cX$ as the composition $R\eqdef N\Phi$.

\begin{prop}
\label{prop:R_properties_multidimensional}
$R$ is a unitary operator from $H$ onto $\cX$.
\end{prop}

\begin{prop}
\label{prop:RK_multidimensional}
Let $\xi\in\Omega$ and $y\in Y$. Then
\begin{equation}
\label{eq:RK_multidimensional}
(RK_{0,y})(\xi)
=
\bigl[\,\conju{q_{j,\xi}(y)}\,\bigr]_{j=1}^{d_\xi}
=
\conju{Q_\xi(y)}.
\end{equation}
\end{prop}

We denote by $\cZ$ the following direct integral of matrix algebras:
\begin{equation}\label{eq:cZ_def}
\cZ\eqdef
\int^{\oplus}_{\Omega}
\bC^{d_\xi\times d_\xi}\,\dif\hnu(\xi).
\end{equation}
Given a matrix family $\sigma=(\sigma(\xi))_{\xi\in\Omega}$ in $\cZ$, let $M_\sigma$ be the ``multiplication operator'' acting in $\cX$ by
\[
(M_\sigma h)(\xi)
\eqdef \sigma(\xi) h(\xi).
\]
Finally, we define $\Lambda\colon \cZ\to\cV$ by
$\Lambda(\sigma)\eqdef R^\ast M_\sigma R$.

\begin{thm}[from shift-invariant operators to matrix families]
$\Lambda$ is an isometric isomorphism of the W*-algebras $\cZ$ and $\cV$.
\end{thm}

\begin{proof}[Idea of the proof]
Follows from Theorem~\ref{thm:general_V_decomposition}, after converting each $\cB(\hH_\xi)$ into $\bC^{d_\xi\times d_\xi}$.
\end{proof}

\begin{cor}
\label{cor:Lambda_inv_explicit_multidimensional}
Let $S\in\cV$.
Then for every $\xi$ in $\Omega$,
\begin{equation}
\label{eq:Lambda_ast_explicit_multidimensional}
(\Lambda^{-1}(S))(\xi)
=\bigl[(R S K_{0,y_1})(\xi),\ldots,(R S K_{0,y_{d_\xi}})(\xi)\bigr]\;
\bigl[\overline{Q_\xi(y_1)},\ldots,\overline{Q_\xi(y_{d_\xi})}\bigr]^{-1},
\end{equation}
where $y_1,\ldots,y_{d_\xi}$ are chosen in $Y$ so that the vectors $Q_\xi(y_1),\ldots,Q_\xi(y_{d_\xi})$ are linearly independent.
\end{cor}

\begin{proof}
Let $\sigma\eqdef\Lambda^{-1}(S)$, i.e., $RS=M_\sigma R$. Then, by~\eqref{eq:RK_multidimensional},
\[
(R S K_{0,y})(\xi)
= (M_\sigma R K_{0,y})(\xi)
= \sigma(\xi) \overline{Q_\xi(y)},
\]
Apply the above equality to the points $y_1,\ldots,y_{d_\xi}$, then join the resulting columns:
\[
\bigl[(R S K_{0,y_1})(\xi),\ldots,(R S K_{0,y_{d_\xi}})(\xi)\bigr]
=\sigma(\xi)
\bigl[\overline{Q_\xi(y_1)},\ldots,\overline{Q_\xi(y_{d_\xi})}\bigr].
\]
Solving this matrix equation for $\sigma(\xi)$ we get~\eqref{eq:Lambda_ast_explicit_multidimensional}.
\end{proof}

\begin{prop}[Berezin transform of a translation-invariant operator]
\label{prop:Ber_invar_oper_multidim}
Let Assumption~\ref{assumption:RKHS} holds,
$S\in\cV$, and let $\sigma\in\cZ$ such that $S=\Lambda(\sigma)$.
Then
\begin{equation}\label{eq:Ber_invar_oper_multidim}
\Ber(S)(x,y)
=\frac{\int_\Omega \sigma(\xi)
L_{\xi,y}(y)\,\dif\hnu(\xi)}{\int_\Omega L_{\xi,y}(y)\,\dif\hnu(\xi)}
\qquad(x\in G,\ y\in Y).
\end{equation}
In particular, $\Ber(S)(x,y)$ does not depend on $x$.
\end{prop}

\begin{proof}
Similar to the proof of Proposition~\ref{prop:Ber_invar_oper}, but applying~\eqref{eq:RK_multidimensional}.
\end{proof}

\begin{prop}[matrix families corresponding to Toeplitz operators with translation-invariant generating symbols]
Let $\psi\in L^\infty(Y)$.
Define $\phi\in L^\infty(G\times Y)$ by $\phi(x,y)=\psi(y)$.
Then $T_\phi=\Lambda(\gamma_\psi)$, where
\begin{equation}\label{eq:gamma_multidimensional}
\gamma_{\psi}(\xi)
\eqdef
\int_Y \psi(v) \conju{Q_\xi(v)}Q_\xi^\ast(v)\,\dif\la(v)
=
\left[
\int_Y \psi(v) \overline{q_{j,\xi}(v)} q_{k,\xi}(v)\,\dif\la(v)
\right]_{j,k=1}^{d_\xi}.
\end{equation}
\end{prop}

\begin{proof}
We will verify that
$R T_\phi R^\ast = M_{\gamma_{\psi}}$.
Same as in the proof of Proposition~\ref{prop:gamma_Toeplitz}, we get~\eqref{eq:FI_P_Mphi}.
If $\xi\in\Omega$ and $v$ in $Y$, then
\[
(M_\phi N^\ast h)(\xi,v)
=\psi(v) Q_\xi^\top(v) h(\xi)
=\sum_{j=1}^{d_\xi}h_j(\xi)q_{j,\xi}(v)\psi(v)
\]
Therefore,
\[
(\hP M_\phi N^\ast h)(\xi,\cdot)
= \hP_\xi\bigl((M_\phi N^\ast h)(\xi,\cdot)\bigr)
= \sum_{k=1}^{d_\xi}h_k(\xi)\hP_\xi (q_{k,\xi}\psi),
\]
and
\begin{align*}
(R T_\phi R^\ast h)(\xi)
&= \bigl[\langle (\hP M_\phi N^\ast h)(\xi,\cdot), q_{j,\xi}\rangle_{L^2(Y)}\bigr]_{j=1}^{d_\xi}
= \left[ \sum_{k=1}^{d_\xi}h_k(\xi) \langle \hP_\xi (q_{k,\xi}\psi),q_{j,\xi}\rangle_{L^2(Y)} \right]_{j=1}^{d_\xi}
\\
&=  \left[ \sum_{k=1}^{d_\xi} \langle  \psi q_{k,\xi},q_{j,\xi}\rangle_{L^2(Y)} h_k(\xi) \right]_{j=1}^{d_\xi}
= \gamma_{\psi}(\xi)h(\xi).
\qedhere
\end{align*}
\end{proof}

\section{Examples}
\label{sec:examples}

To keep this paper to a reasonable length, we restrict ourself to a series of 9 simple examples, mostly with one-dimensional domains $G$ and $Y$.
Example~\ref{example:RBFK} is probably new.
In the other examples, the spectral functions $\gamma_\sigma$ of Toeplitz operators are already known.
Nevertheless, the description of the whole W*-algebra $\cV$ is new for some of these ``old examples''.
We notice that the C*-algebra $\cVT$ from Examples~\ref{example:vertical_harmonic} and \ref{example:radial_harmonic} is not weakly dense in $\cV$.

We use the following notation:
$\mu_n$ is the Lebesgue measure on $\bR^n$ or a subset of $\bR^n$;
$\bR_+\eqdef(0,+\infty)$, $\bN_0\eqdef\{0,1,2,\ldots\}$,
$\bT\eqdef\{\tau\in\bC\colon\ |\tau|=1\}$,
$1_A$ is the characteristic function of $A$; its domain is clear from the context.

In this section, given a LCAG $G$, we denote by $\hG$ a LCAG topologically isomorphic to the dual group of $G$,
and we use some pairing $E\colon G\times\hG\to\bT$.
This means that $\xi\mapsto E(\cdot,\xi)$ is a topological isomorphism
between $\hG$ and the dual group of $G$.
We select the Haar measures $\nu,\hnu$ on $G,\hG$ in such a manner that the Fourier--Plancherel operator $F$ is unitary.
For example, if $G=\bR$, then we put $\hG=\bR$.
One possible pairing is
$E(x,\xi)=\enumber^{\imagunit x\xi}$ with the measures $\nu=\hnu=\frac{1}{\sqrt{2\pi}}\mu_1$;
another one is $E(x,\xi)=\enumber^{2\pi\imagunit x\xi}$, with $\nu=\hnu=\mu_1$.

For each example we have verified assumption~\eqref{eq:K_integral_bounded},
but we have omitted the corresponding computation, for the sake of brevity.

\begin{example}[vertical operators in the holomorphic Bergman space over the upper half-plane]
\label{example:vertical_analytic}
Let $\Pi\eqdef\bR\times\bR_+$ and $H=L^2_{\text{hol}}(\Pi)$.
In this example, $G=\hG=\bR$, $Y=\bR_+$,
$\nu=\hnu=\frac{1}{\sqrt{2\pi}}\mu_1$,
$E(x,\xi)=\enumber^{\imagunit x \xi}$,
$\la=\sqrt{2\pi}\mu_1$,
$\nu\times\la=\mu_2$,
It is well known that $H$ is a Hilbert space with reproducing kernel
\[
K_z(w) = -\frac{1}{\pi(w-\overline{z})^2}.
\]
Identifying $z$ with $(x,y)$ and $w$ with $(u,v)$,
we rewrite the reproducing kernel as
\[
K_{x,y}(u,v)=-\frac{1}{\pi((u-x)+\imagunit(v+y))^2}.
\]
The space $H$ is invariant under horizontal translations.
A simple computation with residues shows that
\[
L_{\xi,y}(v)=\sqrt{\frac{2}{\pi}}\,\xi\,\enumber^{-\xi(y+v)}\,1_{\bR_+}(\xi).
\]
So, in this example $\cV$ is commutative, $\Omega=\bR_+$, and
\[
q_\xi(v)=\left(\frac{2}{\pi}\right)^{1/4}\,
\sqrt{\xi}\,\enumber^{-\xi v}\,1_{\bR_+}(\xi).
\]
Using \eqref{eq:gamma_Toeplitz_def} we compute the spectral functions of vertical Toeplitz operators:
\[
\gamma_\sigma(\xi) = 2\xi \int_{\bR_+} \sigma(v) \enumber^{-2\xi v}\,\dif{}v
\qquad(\xi>0).
\]
This formula coincides with
Vasilevski~\cite[Theorem 3.1]{Vasilevski1999BergmanToeplitz} and~\cite[Theorem 5.2.1]{Vasilevski2008book},
see also Grudsky, Karapetyants, and Vasilevski~\cite{GrudskyKarapetyantsVasilevski2004par}.
The C*-algebra $\cG$ in this example consists of all bounded functions on $\bR_+$, uniformly continuous with respect to the $\log$-distance, see~\cite{HerreraYanezMaximenkoVasilevski2013,HerreraYanezHutnikMaximenko2014}.
\end{example}

\begin{example}[vertical operators in the harmonic Bergman space over the upper half-plane]
\label{example:vertical_harmonic}
Let $G$, $Y$, $\nu$, $\la$, and $E$ be the same as in Example~\ref{example:vertical_analytic},
but $H\eqdef L^2_{\text{harm}}(\Pi)$
be the Bergman space of harmonic functions on $\Pi$.
Using Riesz theorem about the Hardy spaces of harmonic functions,
one can show that
$L^2_{\text{harm}}(\Pi)
=L^2_{\text{hol}}(\Pi)
\oplus\overline{L^2_{\text{hol}}(\Pi)}$.
Therefore, $H$ is a RKHS with reproducing kernel
\[
K_z(w) = -\frac{1}{\pi(w-\overline{z})^2} - \frac{1}{\pi(\overline{w}-z)^2}.
\]
Identifying $z$ with $(x,y)$ and $w$ with $(u,v)$, we obtain
\[
K_{x,y}(u,v)
=-\frac{1}{\pi((u-x)+\imagunit(v+y))^2}
-\frac{1}{\pi((u-x)-\imagunit(v+y))^2}.
\]
Now
\[
L_{\xi,y}(v)=\sqrt{\frac{2}{\pi}}\,|\xi|\,\enumber^{-|\xi|(y+v)}\qquad(\xi\in\bR).
\]
We conclude that in this example $\cV$ is commutative, $\Omega=\bR\setminus\{0\}$,
\[
q_\xi(v)
=\left(\frac{2}{\pi}\right)^{1/4}\,\sqrt{|\xi|}\,\enumber^{-2|\xi| v},
\]
and
\[
\gamma_\sigma(\xi)
=\gamma_\sigma(|\xi|)
=2 |\xi| \int_{\bR_+} \sigma(v) \enumber^{-2|\xi| v}\,\dif{}v.
\]
Thereby we reproduce a result by Loaiza and Lozano~\cite[Theorem 4.16]{LoaizaLozano2013}.
In this example, the spectral functions $\gamma_\sigma$ are even.
The C*-algebra $\cG$ generated by $\cG_0$ coincides with the closure of $\cG_0$ in the norm topology and consists of all even function on $\bR\setminus\{0\}$
whose restrictions to $\bR_+$ are uniformly continuous with respect to the $\log$-distance.
By Theorem~\ref{thm:SW_for_Wstar_alg}, the W*-algebra generated by $\cG_0$ is the class of all essentially bounded even functions on $\bR$,
which is a proper subset of $L^\infty(\bR)$.
So, $\cG$ is not $\tau_\Omega$-dense in $L^\infty(\Omega)$.
By Corollary~\ref{cor:cVT_and_cG}, this means that $\cVT$ is not weakly dense in $\cV$.
\end{example}

\begin{example}[vertical operators in the Bergman space
of true-polyanalytic functions over the upper half-plane]
\label{example:vertical_true_polyanalytic}
Let $G$, $Y$, $\nu$, $\hnu$, $\la$, $E$
be the same as in Example~\ref{example:vertical_analytic}.
For a fixed $m$ in $\bN$,
we consider the space $H\eqdef L^2_{(m)\text{-hol}}(\Pi)$
of all square-integrable $m$-true-polyanalytic functions on the upper half-plane $\Pi$.
Applying the Fourier transform to
the differential equation defining $H$,
Vasilevski computed~\cite[Section 3.4]{Vasilevski2008book}
the operator $(F\otimes I)P(F\otimes I)^\ast$ which we denote by $\hP$.
Namely, he proved that $\hP$ acts by~\eqref{eq:hP_via_L}, with
\begin{equation}\label{eq:L_vertical_true_polyanalytic}
L_{\xi,y}(v)
= 
1_{\bR_+}(\xi)
\sqrt{\frac{2}{\pi}}\,
\xi \enumber^{-\xi(y+v)}
L_{m-1}(2\xi y) L_{m-1}(2\xi v),
\end{equation}
where $L_k$ is the Laguerre polynomial of degree $k$.
This means that $\Omega=\bR_+$,
\[
q_\xi(v) = (2/\pi)^{1/4} \sqrt{\xi} \enumber^{-\xi v} L_{m-1}(2\xi v)\,1_{\bR_+}(\xi)\qquad(\xi\in\bR,\ v>0),
\]
and
\begin{equation}\label{eq:gamma_vertical_true_polyanalytic}
\gamma_\sigma(\xi) = 2\xi \int_{\bR_+} \sigma(v) \enumber^{-2\xi v} (L_{m-1}(2\xi v))^2\,\dif{}v.
\end{equation}
Formula~\eqref{eq:gamma_vertical_true_polyanalytic} was found by Hutn\'{i}k~\cite[Theorem~3.2]{Hutnik2011}
and by Ram\'{i}rez-Ortega and S\'{a}nchez-Nungaray~\cite[Theorem~3.2]{RamirezOrtegaSanchezNungaray2015}.
The C*-algebra $\cG$ for this example coincides with the C*-algebra $\cG$ from Example~\ref{example:vertical_analytic}, see~\cite{HutnikMaximenkoMiskova2016}.
Vasilevski noticed~\cite[Theorem~3.4.1]{Vasilevski2008book}
that the reproducing kernel of $L^2_{(m)\text{-hol}}(\Pi)$ can be obtained by applying $(F\otimes I)^\ast$ to $L$ given by~\eqref{eq:L_vertical_true_polyanalytic}.
Using explicit expressions for the Laguerre polynomials one obtains
\begin{equation}
\label{eq:K_true_poly_half_plane}
K_z(w)
=-\frac{1}{(w-\overline{z})^2}
\sum_{j,k=0}^{m-1}
(-1)^{j+k}
\frac{(m-1)!\,(j+k+1)!}%
{(j!\,k!)^2\,(m-1-j)!\,(m-1-k)!}
\;
\frac{(w-\overline{w})^j\,(z-\overline{z})^k}{(w-\overline{z})^{j+k}}.
\end{equation}
\end{example}

\begin{example}[vertical operators in the Bergman space of polyanalytic functions over the upper half-plane]
Here $G$ and $Y$ are the same as in Example~\ref{example:vertical_true_polyanalytic},
and $H=L_{n\text{-hol}}^2(\Pi)$ 
is the space of square-integrable $n$-analytic functions on $\Pi$.
The decomposition
$H=\cH_1\oplus\cdots\oplus\cH_n$,
where $\cH_m$ is the space from Example~\ref{example:vertical_true_polyanalytic}, implies that
\begin{equation}\label{eq:L_poly_analytic}
L_{\xi,y}(v)
= 1_{\bR_+}(\xi)
\sqrt{\frac{2}{\pi}}\,
\xi \enumber^{-\xi(y+v)}
\sum_{m=1}^n L_{m-1}(2\xi y) L_{m-1}(2\xi v).
\end{equation}
It would be interesting to prove~\eqref{eq:L_poly_analytic} directly, applying the Fourier transform to the reproducing kernel of $L_{n\text{-hol}}^2(\Pi)$,
computed in~\cite{Pessoa2013}
and~\cite{LealMaxRamos2021} in terms of Jacobi polynomials:
\begin{equation}\label{eq:K_polyanalytic}
K^{H}_z(w)
=\frac{n\,(-1)^n}{\pi}
\frac{(z-\overline{w})^{n-1}}{(w-\overline{z})^{n+1}}\,
P_{n-1}^{(0,1)}%
\left(2\,\frac{|w-z|^2}{|w-\overline{z}|^2}-1\right).
\end{equation}
The orthogonality of the Laguerre polynomials implies that
\eqref{eq:L_poly_analytic} is a particular case of~\eqref{eq:L_via_xis}, with $\Omega=\bR_+$, $d_\xi=n$, and
\[
q_{j,\xi}(v)
=(2/\pi)^{1/4}\,
\sqrt{\xi}
\enumber^{-\xi v} L_{j-1}(2\xi v)
\qquad(j=1,\ldots,n,\ \xi>0,\ v>0).
\]
Thereby, the W*-algebra $\cV$ in this example is spatially isomorphic to the direct integral
of matrix algebras,
\[
\cV \cong
\int_{\bR_+}^{\oplus}\bC^{n\times n}\,\dif\hnu(\xi)
\cong
L^\infty(\bR_+,\bC^{n\times n}).
\]
Ram\'{i}rez-Ortega and S\'{a}nchez--Nungaray~\cite[Theorem~4.7]{RamirezOrtegaSanchezNungaray2015} found a complete description of a certain non-commutative C*-subalgebra of $\cVT$.
\end{example}

\begin{example}[vertical operators in wavelet spaces over the positive affine group]
\label{example:wavelet}
Let $\psi$ be a wavelet of the class $L^2(\bR)$ satisfying the admissibility condition:
\begin{equation}\label{eq:admissibility}
\int_{\bR_+}|(F\psi)(t\xi)|^2\,\frac{\dif{}t}{t}=1
\qquad(\xi\in\bR\setminus\{0\}),\qquad
(F\psi)(0)=0.
\end{equation}
Put $G=\bR$, $\nu=\hnu=\mu_1$,
$E(x,\xi)=\enumber^{2\pi\imagunit x\xi}$,
$Y=\bR_+$, $\dif\la(y)=\frac{\dif{}y}{y^2}$.
Notice that $G\times Y$ can be identified with the positive affine group.
For every $(x,y)$ in $G\times Y$, put
\[
\psi_{x,y}(t)=\frac{1}{\sqrt{y}}\psi\left(\frac{t-x}{y}\right).
\]
Define $W_\psi\colon L^2(\bR)\to L^2(G\times Y)$ by
\[
(W_\psi f)(x,y) \eqdef \langle f, \psi_{x,y}\rangle_{L^2(\bR)}.
\]
The wavelet space $H$ associated with $\psi$ can be defined as $W_\psi(L^2(\bR))$.
It is a RKHS over $G\times Y$, with reproducing kernel
\[
K_{x,y}(u,v)
=\langle \psi_{u,v}, \psi_{x,y}\rangle_{L^2(\bR)}
=\langle \psi_{u-x,v}, \psi_{0,y}\rangle_{L^2(\bR)}.
\]
Then
\[
L_{\xi,y}(v)=\sqrt{\vphantom{t}yv}\,(F\psi)(y\xi)\,\overline{(F\psi)(v\xi)}.
\]
So, in this example $\Omega=\bR$, $\cV$ is commutative, and
\[
q_\xi(v)=\sqrt{v}\,\overline{(F\psi)(v\xi)}.
\]
The property $\|q_\xi\|_{L^2(Y)}=1$ follows from \eqref{eq:admissibility}.
The spectral functions are given by
\[
\gamma_\sigma(\xi)
=\int_{\bR_+} \sigma(v)\,|(F\psi)(v\xi)|^2\,\frac{\dif{}v}{v}.
\]
This formula was found by Hutn\'{i}k and Hutn\'{i}kov\'{a} \cite{HutnikHutnikova2011}.
\end{example}

Let us mention without further details another similar example,
studied by Hutn\'{i}kov\'{a} and Mi\'{s}kov\'{a}~\cite{HutnikovaMiskova2015}:
translation-invariant operators in the space related to the continuous Stockwell transform.

In some examples, it is convenient to transform
the domain of the functions and the RKHS.
The next simple proposition provides a recipe
to compute the reproducing kernel after a change of variables followed by the multiplication by some weight.

\begin{prop}\label{prop:transform_RKHS}
Let $D_1$ and $D_2$ be some non-empty sets,
$\cH_1$ be a RKHS over $D_1$, with reproducing kernel $(K^{\cH_1}_z)_{z\in D_1}$,
and $\cH_2$ be a complex vector space of functions over $D_2$, with a pre-inner product.
Suppose that $A$ is a linear isometry from $\cH_1$ to $\cH_2$, acting by the rule
\[
(Af)(z)=p(z)f(\phi(z))\qquad(z\in D_2,\ f\in\cH_1),
\]
where $\phi\colon D_2\to D_1$ and $p\colon D_2\to\bC$.
Then $H\eqdef A(\cH_1)$ is a RKHS over $D_2$, and the reproducing kernel in $H$ can be computed by
\begin{equation}\label{eq:change_RK}
K^H_z(w) = \conju{p(z)} K^{\cH_1}_{\phi(z)}(\phi(w)) p(w).
\end{equation}
\end{prop}

\begin{proof}
Since $A$ is a linear isometry and $\cH_1$ is a Hilbert space, $H$ is also a Hilbert space.
The rest of the proof is the same as in~\cite[Proposition~4.3]{LealMaxRamos2021}.
A similar construction is explained in~\cite[Section~2.6]{AglerMcCarthy2002}.
\end{proof}

\begin{example}[radial operators in the analytic Bergman space over the unit disk]
\label{example:radial_analytic}
Let $\cH_1=L^2_{\text{hol}}(\bD)$ be the Bergman space of analytic functions
over the unit disk $\bD$ provided with the plane Lebesgue measure $\mu_2$.
It is well known that the reproducing kernel of $\cH_1$ is
\[
K^{\cH_1}_z(w)
=\frac{1}{\pi(1-\overline{z}w)^2}.
\]
Let $G$ be the group $\bR/(2\pi\bZ)$ with the normalized Haar measure $\nu$
(we identify $G$ with $[0,2\pi)$),
$\hG=\bZ$ with the counting measure $\hnu$,
$E(u+2\pi\bZ,\xi)=\enumber^{\imagunit u\xi}$ for $u\in\bR$ and $\xi$ in $\bZ$,
and $Y$ be the interval $[0,1)$ with the measure $\dif\la(v)=v\,\dif{}v$.
Define $\phi\colon G\times Y\to\bD$ and $p\colon G\times Y\to\bC$ by
\[
\phi(u,v)=v \enumber^{\imagunit u},\qquad p(u,v)=\sqrt{2\pi}.
\]
Let $\cH_2 = L^2(G\times Y,\nu\otimes\la)$.
The operator $A$, defined as Proposition~\ref{prop:transform_RKHS}, is a linear isometry:
\[
    \|Af\|^2_{\cH_2} 
    = \int_0^1 \int_0^{2\pi} |f(v\enumber^{\imagunit u})|^2 \,v\,\dif{}u\,\dif{}v\\
    = \int_{\bD} |f(z)|^2 \dif{}\mu_2(z) = \|f\|^2_{\cH_1}.
\]
Hence, $A$ converts $\cH_1$
into a certain RKHS $H$ over $G\times Y$, with reproducing kernel
\[
K_{x,y}(u,v)
=\overline{p(x,y)}\,K^{\cH_1}_{\phi(x,y)}(\phi(u,v))\,p(u,v)
=\frac{2}{(1-yv\enumber^{\imagunit(u-x)})^2}.
\]
Obviously, $A$ intertwines the rotation operators acting in $\cH_1$
into ``horizontal translations'' acting in $H$.
Now we notice that the function $K_{0,y}(\cdot\,,v)$ decomposes into the Fourier series
\[
K_{0,y}(u,v) 
= \sum_{\xi=0}^\infty 2(\xi+1)\,(yv)^\xi\,\enumber^{\imagunit \xi u},
\]
which means that its  Fourier coefficients are
\[
L_{\xi,y}(v)
=2(\xi+1) (yv)^\xi\,1_{\bN_0}(\xi).
\]
Thus, in this example, $\Omega=\bN_0$ and $q_\xi(v)=\sqrt{2(\xi+1)}\,v^\xi$.
The W*-algebra of radial operators in $\cH_1$ is commutative,
and the sequence of the eigenvalues of a radial Toeplitz operator is computed by
\[
\gamma_\sigma(\xi)
=2(\xi+1) \int_0^1 \sigma(v)\,v^{2\xi+1}\,\dif{}v
=(\xi+1) \int_0^1 \sigma(\sqrt{r})\,r^\xi\,\dif{}r
\qquad(\xi\in\bN_0).
\]
These results are well known and easily obtained from the fact
that the radial operators are diagonal in the monomial basis
$(\sqrt{(\xi+1)/\pi}\,z^\xi)_{\xi=0}^\infty$.
Our treatment of this example is close to
\cite[Chapters 4, 6]{Vasilevski2008book}
and~\cite{GrudskyKarapetyantsVasilevski2004rad},
where $L^2(\bD,\mu_2)$ is decomposed into
$L^2(\bR/(2\pi\bZ))\otimes L^2([0,1),r\,\dif r)$,
and the Fourier transform over $\bR/(2\pi\bZ)$
is applied to the equation defining $\cH_1$.
The C*-algebra $\cVT$ for this example was described in~\cite{GrudskyMaximenkoVasilevski2013} using Su\'{a}rez~\cite{Suarez2008}.
\end{example}

Radial operators in the Segal--Bargmann--Fock space on $\bC$
can be studied similarly to Example~\ref{example:radial_analytic}.
Moreover, Example~\ref{example:radial_analytic} is easily generalized
to the case of separately radial operators acting
on the Bergman space over the unit ball in $\bC^n$.
In that case $G=(\bR/(2\pi\bZ))^n$ and $\Omega=\bN_0^n$.

\begin{example}[radial operators in the harmonic Bergman space over the unit disk]
\label{example:radial_harmonic}
For $\cH_1=L^2_{\text{harm}}(\bD)$,
\[
K^{\cH_1}_z(w)
=\frac{1}{\pi(1-\overline{z}w)^2}
+\frac{1}{\pi(1-\overline{w}z)^2}
-1.
\]
Similarly to Example~\ref{example:radial_analytic},
after passing to the polar coordinates and computing the Fourier coefficients, we have
\[
L_{\xi,y}(v)=2(|\xi|+1) (yv)^{|\xi|}\qquad(\xi\in\bZ,\ y,v\in[0,1)).
\]
The W*-algebra of radial operators in $L_{\text{harm}}^2(\bD)$ is commutative,
$\Omega=\bZ$, $q_\xi(v)=\sqrt{2(|\xi|+1)}\,v^{|\xi|}$, and
\begin{equation}\label{eq:eigenvalues_radial_harmonic}
\gamma_\sigma(\xi)=(|\xi|+1) \int_0^1 \sigma(\sqrt{r})\,r^{|\xi|}\,\dif{}r.
\end{equation}
Formula \eqref{eq:eigenvalues_radial_harmonic}
was previously obtained by Loaiza and Lozano~\cite[Theorem~3.4]{LoaizaLozano2013}.

Similarly to Example~\ref{example:vertical_harmonic}, the symmetry of formula~\eqref{eq:eigenvalues_radial_harmonic} with respect to the sign of $\xi$ implies that
$\cG$ is a subclass of bounded symmetric sequences.
By Corollary~\ref{cor:cVT_and_cG}, the C*-algebra generated by Toeplitz operators with radial symbols is not weakly dense in the W*-algebra of all bounded radial operators on $L^2_{\text{harm}}(\bD)$.
\end{example}

\begin{rem}
\label{rem:Toeplitz_harmonic_not_weakly_dense}
Since the radialization transform of bounded linear operators in $L^2_{\text{harm}}(\bD)$
is continuous in WOT and converts Toeplitz operators into radial Toeplitz operators, the last paragraph of Example~\ref{example:radial_harmonic}
implies that the set of all Toeplitz operators is not weakly dense in $\cB(L^2_{\text{harm}}(\bD))$. This result was proven more directly in~\cite{BMR2021}.
In contrast, the weak density of Toeplitz operators $\cB(L^2_{\text{hol}}(\bD))$ has already been proven by Engli\v{s}~\cite{Englis1992}.
\end{rem}

\begin{example}[angular operators in the analytic Bergman space over the upper half-plane]
\label{example:angular_analytic}
Let $\cH_1=L^2_{\text{hol}}(\Pi)$.
We say that an operator $A$ of the class $\cB(\cH_1)$ is \emph{angular}
if $A$ commutes with all dilations $D_h$ ($h>0$), where $D_h$ is given by
\[
(D_h f)(w)=h^{-1} f(h^{-1} w).
\]
Let $G=\bR$, $Y=(0,\pi)$,
$\nu=\hnu=\frac{1}{\sqrt{2\pi}}\mu_1$,
$E(x,\xi)=\enumber^{\imagunit x \xi}$,
and $\la$ be the Lebesgue measure on $(0,\pi)$.
Define $\phi\colon G\times Y\to\Pi$, $p\colon G\times Y\to\bC$,
and $A\colon\cH_1\to L^2(G\times Y)$ by
\[
\phi(u,v)\eqdef \enumber^{u+\imagunit v},\quad
p(u,v)\eqdef
(2\pi)^{1/4}
\enumber^{u+\imagunit v},\quad
(Af)(u,v)
=(2\pi)^{1/4}
\enumber^{u+\imagunit v}
f(\enumber^{u+\imagunit v}).
\]
It is easy to see that $A$ is a linear isometry,
so we can apply Proposition~\ref{prop:transform_RKHS}.
The space $H\eqdef A(\cH_1)$ has reproducing kernel
\[
K_{x,y}(u,v)
=-\sqrt{\frac{2}{\pi}}\frac{\enumber^{x-\imagunit y}\,\enumber^{u+\imagunit v}}{\left(\enumber^{u+\imagunit v}-\enumber^{x-\imagunit y}\right)^2}
=-\sqrt{\frac{2}{\pi}}%
\frac{1}%
{4\left(\sinh\frac{u-x+\imagunit(v+y)}{2}\right)^2}.
\]
The linear isometry $A$ intertwines the dilations, acting in $\cH_1$, with the horizontal translations, acting in $H$:
\[
A D_{\enumber^a} A^\ast
=\rho_H(a).
\]
Hence, the algebra of angular operators in $\cH_1$ is converted into $\cV$.
For $\xi>0$, integral~\eqref{eq:L_definition} can be computed via the residues at the points
$u_k\eqdef-\imagunit(v+y)-2\pi\imagunit k$, $k\in\bN_0$:
\begin{align*}
L_{\xi,y}(v)
&=-\frac{1}{4\pi}
\int_{\bR}
\frac{\enumber^{-\imagunit\xi u}\,\dif{}u}%
{\left(\sinh\frac{u+\imagunit(v+y)}{2}\right)^2}
=\frac{\imagunit}{2}\,
\sum_{k=0}^\infty
\operatornamewithlimits{res}_{u=u_k}
\frac{\enumber^{-\imagunit\xi u}}%
{\left(\sinh\frac{u+\imagunit(v+y)}{2}\right)^2}
\\
&=
\frac{\imagunit}{2}
\sum_{k=0}^\infty
\left(-4\imagunit\xi
\enumber^{-\xi(v+y)}
\enumber^{-2k\pi\xi}
\right)
=
\frac{2\xi\,\enumber^{-\xi(y+v)}}{1-\enumber^{-2\pi \xi}}.
\end{align*}
For $\xi<0$, the integral expresses through the residues at the points $u_k$ with $k<0$, but the final formula for $L_{\xi,y}(v)$ is the same.
We conclude that $\cV$ is commutative, $\Omega=\bR$, and
\[
q_\xi(v)
=\sqrt{\frac{2\xi}{1-\enumber^{-2\pi \xi}}}
\;\enumber^{-\xi v}.
\]
The spectral functions of angular Toeplitz operators can be computed by
\[
\gamma_\sigma(\xi)
=
\frac{2\xi}{1-\enumber^{-2\pi \xi}}
\int_0^\pi \sigma(v)\,\enumber^{-2\xi v}\,\dif{}v.
\]
This formula coincides with
\cite[Theorem 7.2.1]{Vasilevski2008book},
see also~\cite{GrudskyKarapetyantsVasilevski2004hyp}.
The C*-algebra $\cG$ for this example is found by Esmeral, Maximenko, and Vasilevski~\cite{EsmeralMaximenkoVasilevski2015}.
\end{example}

\begin{example}[vertical operators in the RKHS asssociated to the complex Gaussian kernel]
\label{example:RBFK}
The following reproducing kernel
and its restriction to $\bR^n$
are extensively used in machine learning.
They are known as the Gaussian kernel or the 
radial basis function kernel.
\[
K_{z}(w) = \exp\left(-\alpha^2\sum_{j=1}^n(z_j-\conju{w_j})^2\right)\qquad(z,w\in\bC^n).
\]
Here $\alpha$ is a fixed positive number.
Steinwart, Hush, and Scovel~\cite{SteinwartHushScovel2006} proved that the corresponding RKHS is
$H=\{f\in\operatorname{Hol}(\bC^n)\colon\ \|f\|_{\operatorname{RBFK}}<+\infty\}$,
where
\[
\|f\|_{\operatorname{RBFK}}\eqdef
\left(
\frac{2^n\alpha^{2n}}{\pi^n}
\int_{\bC^n} |f(z)|^2
\exp\left(-4\alpha^2\sum_{j=1}^n\Im(z_j)^2\right)
\dif{}\mu_{2n}(z) \right)^{1/2}.
\]
We identify the domain $\bC^n$ with $G\times Y$, where $G=Y=\bR^n$.
The measures and the pairing are
\[
\nu=\hnu=\mu_n,\qquad
\dif\la(v)=\frac{2^n\alpha^{2n}}{\pi^n}\,\exp\left(-4\alpha^2\|v\|^2\right),\qquad
E(x,y)=\exp\left(2\pi\imagunit \langle x,y\rangle\right).
\]
Then the kernel takes the form
\[
K_{x,y}(u,v) 
= \exp\left(
-\alpha^2\sum_{j=1}^n ((u_j-x_j)^2-(v_j+y_j)^2
+2\imagunit(u_j-x_j)(v_j+y_j))
\right).
\]
The computation of $L_{\xi,y}(v)$ can be reduced to the Gaussian integral and results in
\[
L_{\xi,y}(v)=
  \left(\frac{\sqrt{\pi}}{\alpha}\right)^n
  \exp\left(-\sum_{j=1}^n \left(2\pi(v_j+y_j)\xi_j+\frac{\pi^2\xi_j^2}{\alpha^2}\right)\right).
\]
In this example, $\Omega=\bR^n$,
$\cV$ is commutative, and
\[
q_{\xi}(v)
= \left(\frac{\sqrt{\pi}}{\alpha}\right)^{n/2} \exp\left(
-\sum_{j=1}^n \left(2\pi v_j\xi_j + \frac{\pi^2\xi_j^2}{2\alpha^2}\right)
\right).
\]
\end{example}

\begin{rem}
For each example, we tested the equality $((F\otimes I)K_{0,y})(\xi,v)=L_{\xi,y}(v)$
numerically in Sagemath. 
In Example~\ref{example:wavelet}, we used the Mexican hat wavelet.
\end{rem}

\section*{Fundations}

The authors have been partially supported by Proyecto CONACYT ``Ciencia de Frontera''
FORDECYT-PRONACES/61517/2020, by CONACYT (Mexico) scholarships, and by IPN-SIP projects
(Instituto Politécnico Nacional, Mexico).

\section*{Acknowledgements}

The authors are grateful to Nikolai L. Vasilevski
who introduced to us the world of commutative C*-algebras
of translation-invariant Toeplitz operators acting in various reproducing kernel Hilbert spaces,
to Christian Rene Leal Pacheco for the joint revision of various parts of this paper,
to Matthew G. Dawson for explaining us some ideas from Section~\ref{sec:invar_L2},
to Yuri Latushkin for the advice to use tensor products in Section~\ref{sec:invar_L2}, and to Gestur \'{O}lafsson for indicating us that the embedding $C_b(X)\subseteq L^\infty(X,\mu)$ in Section~\ref{sec:Wstar_Stone_Weierstrass} required the assumption $\operatorname{support}(\mu)=X$.

\bigskip\noindent
Crispin Herrera-Ya\~{n}ez\newline
Instituto Polit\'{e}cnico Nacional,
Escuela Superior de C\'{o}mputo\newline
C\'odigo Postal 07730,
Ciudad de M\'{e}xico,
Mexico\newline
e-mail: cherreray@ipn.mx\newline
https://orcid.org/0000-0002-3339-657X

\bigskip\noindent
Egor A. Maximenko\newline
Instituto Polit\'{e}cnico Nacional,
Escuela Superior de F\'{i}sica y Matem\'{a}ticas\newline
C\'odigo Postal 07730,
Ciudad de M\'{e}xico,
Mexico\newline
e-mail: egormaximenko@gmail.com\newline
https://orcid.org/0000-0002-1497-4338

\bigskip\noindent
Gerardo Ramos-Vazquez\newline
Centro de Investigaci\'{o}n y de Estudios Avanzados
del Instituto Polit\'{e}cnico Nacional,
Departamento de Matem\'{a}ticas\newline
C\'odigo Postal 07360,
Ciudad de M\'{e}xico,
Mexico\newline
e-mail: ger.ramosv@gmail.com\newline
https://orcid.org/0000-0001-9363-8043


\begin{thebibliography}{11}

\bibitem{AglerMcCarthy2002}
Agler, J.;
McCarthy, J.E.:
Pick Interpolation and Hilbert Function Spaces.
American Mathematical Society,
Providence, RI
(2002).

\bibitem{Aronszajn1950}
Aronszajn, N.:
Theory of reproducing kernels. Transactions of the American Mathematical Society, \textbf{68}, 337--404 (1950).
\doi{10.2307/1990404}

\bibitem{BMR2021}
Barrera-Castel\'{a}n, R.M.; Maximenko, E.A.; Ramos-Vazquez, G.:
Radial operators on polyanalytic weighted Bergman spaces.
Bol. Soc. Mat. Mex. \textbf{27}, 43 (2021).
\doi{10.1007/s40590-021-00348-w}

\bibitem{BauerFulshe2020}
Bauer,~W.; Fulsche,~R.:
Berger-Coburn theorem, localized operators, and the Toeplitz algebra. 
In: Bauer,~W.; Duduchava,~R.; Grudsky,~S.; Kaashoek,~M. (eds.) Operator Algebras, Toeplitz Operators and Related Topics. Operator Theory: Advances and Applications, vol.~279. Birkhauser, Cham (2020).
\doi{10.1007/978-3-030-44651-2\_8}

\bibitem{DawsonOlafssonQuiroga2015}
Dawson,~M.; \'{O}lafsson,~G.; Quiroga-Barranco,~R.:
Commuting Toeplitz operators on bounded symmetric domains
and multiplicity-free restrictions of holomorphic discrete series.
J. Funct. Anal. \textbf{268}, 1711--1732 (2015).
\doi{10.1016/j.jfa.2014.12.002}

\bibitem{DawsonOlafssonQuiroga2018}
Dawson, M.; \'{O}lafsson, G.; Quiroga-Barranco, R.:
The restriction principle and commuting families of Toeplitz operators on the unit ball.  
S\~{a}o Paulo J. Math. Sci. \textbf{12}, 196--226 (2018). \doi{10.1007/s40863-018-0104-1}

\bibitem{Dixmier1981vonNeumann}
Dixmier, J.:
Von Neumann Algebras.
North-Holland Publishing Company,
Amsterdam, New York, Oxford
(1981).

\bibitem{Englis1992}
Engli\v{s}, M.:
Density of algebras generated by Toeplitz operator on Bergman spaces,
Ark. Mat. \textbf{30}, 227--243  (1992).
\doi{10.1007/BF02384872}

\bibitem{EsmeralMaximenkoVasilevski2015}
Esmeral, K.; Maximenko, E.A.; Vasilevski, N.:
C*-algebra generated by angular Toeplitz operators on the weighted Bergman spaces over the upper half-plane.
Integr. Equ. Oper. Theory \textbf{83}, 413--428  (2015).
\doi{10.1007/s00020-015-2243-4}

\bibitem{Folland2016harmonic}
Folland, G.B.:
A Course in Abstract Harmonic Analysis. 2nd ed.
Taylor \& Francis,
Boca Raton, Florida (2016).

\bibitem{Folland1999real}
Folland, G.B.:
Real Analysis: Modern Techniques and Their Applications, 2nd ed.
Wiley, New York (1999).

\bibitem{GrudskyKarapetyantsVasilevski2004par}
Grudsky, S.; Karapetyants, A.; Vasilevski, N.:
Dynamics of properties of Toeplitz operators on the upper half-plane: Parabolic case.
J. Operator Theory \textbf{52}, 185--214 (2004).
\myurl{http://www.jstor.org/stable/24718968}

\bibitem{GrudskyKarapetyantsVasilevski2004hyp}
Grudsky, S.; Karapetyants, A.; Vasilevski, N.:
Dynamics of properties of Toeplitz operators on the upper half-plane: Hyperbolic case.
Bol. Soc. Mat. Mexicana \textbf{10}, number 1, 119--138 (2004).
\myurl{https://www.smm.org.mx/boletin\_anterior/v10/n1.pdf}

\bibitem{GrudskyKarapetyantsVasilevski2004rad}
Grudsky, S.; Karapetyants, A.; Vasilevski, N.:
Dynamics of properties of Toeplitz operators with radial symbols.
Integr. Equ. Oper. Theory \textbf{50}, 217--253 (2004).
\doi{10.1007/s00020-003-1295-z}

\bibitem{GrudskyQuirogaVasilevski2006}
Grudsky, S.; Quiroga-Barranco, R.; Vasilevski, N.:
Commutative C*-algebras of Toeplitz operators and quantization on the unit disk.
J. Funct. Anal. \textbf{234}, 1--44  (2006).
\doi{10.1016/j.jfa.2005.11.015}

\bibitem{GrudskyMaximenkoVasilevski2013}
Grudsky, S.M.; Maximenko, E.A.; Vasilevski, N.L.:
Radial Toeplitz operators on the unit ball and slowly oscillating sequences.
Commun. Math. Anal. \textbf{14}, 77--94 (2013).
\myurl{http://projecteuclid.org/euclid.cma/1356039033}

\bibitem{Hagger2020}
Hagger, R.:
Essential commutants and characterizations of the
Toeplitz algebra.
Preprint:
\href{https://arxiv.org/abs/2002.02344}{arXiv:2002.02344}
[math.FA] (2020).

\bibitem{HerreraYanezMaximenkoVasilevski2013}
Herrera Ya\~{n}ez, C.; Maximenko, E.A.; Vasilevski, N.:
Vertical Toeplitz operators on the upper half-plane and very slowly oscillating functions.
Integr. Equ. Oper. Theory \textbf{77}, 149--166  (2013).
\doi{10.1007/s00020-013-2081-1}

\bibitem{HerreraYanezHutnikMaximenko2014}
Herrera Ya\~{n}ez, C.; Hutn\'{i}k, O.; Maximenko, E.A.:
Vertical symbols, Toeplitz operators on weighted Bergman spaces over the upper half-plane and very slowly oscillating functions.
Comptes Rendus Mathematique \textbf{352}, 129--132  (2014).
\doi{10.1016/j.crma.2013.12.004}

\bibitem{HewittRoss1979}
Hewitt, E; Ross, K.A.:
Abstract Harmonic Analysis I. 2nd ed.
Springer, New York (1979).

\bibitem{Hutnik2011}
Hutn\'{i}k, O.:
Wavelets from Laguerre polynomials and Toeplitz-type operators.
Integr. Equ. Oper. Theory
\textbf{71}, 357--388 (2011).
\doi{10.1007/s00020-011-1907-y}

\bibitem{HutnikHutnikova2011}
Hutn\'{i}k, O.; Hutn\'{i}kov\'{a}, M.:
On Toeplitz localization operators.
Arch. Math. \textbf{97}, 333--344  (2011).
\doi{10.1007/s00013-011-0307-5}

\bibitem{HutnikovaMiskova2015}
Hutn\'{i}kov\'{a}, M.; Mi\'{s}kov\'{a}, A.:
Continuous Stockwell transform: Coherent states and localization operators.
J. Math. Phys. \textbf{56}, 073504  (2015).
\doi{10.1063/1.4926950}

\bibitem{HutnikMaximenkoMiskova2016}
Hutn\'{i}k, O; Maximenko, E.; Mi\v{s}kov\'{a} A.:
Toeplitz localization operators: spectral functions density.
Complex Anal. Oper. Theory \textbf{10}, 1757--1774  (2016).
\doi{10.1007/s11785-016-0564-1}



\bibitem{Larsen1971}
Larsen, R.:
An Introduction to the Theory of Multipliers.
Springer, Berlin, Heidelberg  (1971).
\doi{10.1007/978-3-642-65030-7}

\bibitem{LealMaxRamos2021}
Leal-Pacheco, C.R.; Maximenko, E.A.; Ramos-Vazquez, G.:
Homogeneously polyanalytic kernels on the unit ball and the Siegel domain.
Complex Anal. Oper. Theory \textbf{15}, 99 (2021).
\doi{10.1007/s11785-021-01145-z}

\bibitem{LoaizaLozano2013}
Loaiza, M.; Lozano, C.:
On C*-algebras of Toeplitz operators on the harmonic Bergman space.
Integr. Equ. Oper. Theory \textbf{76}, 105--130  (2013).
\doi{10.1007/s00020-013-2046-4}

\bibitem{MaximenkoTelleria2020}
Maximenko, E.A.;
Teller\'{i}a-Romero, A.M.:
Radial operators in polyanalytic Bargmann--Segal--Fock spaces.
Chapter in the book: Bauer, W.; Duduchava, R.; Grudsky, S.; Kaashoek, M. (eds.) Operator Algebras, Toeplitz Operators and Related Topics, pp. 277--305.
Book series Operator Theory: Advances and Applications, vol. 279.
Birkhäuser, Cham (2020).
\doi{10.1007/978-3-030-44651-2\_18} 

\bibitem{Pessoa2013}
Pessoa, L.V.:
The method of variation of the domain for poly-Bergman spaces. Math. Nachr. \textbf{286}, 1850--1862 (2013). \doi{10.1002/mana.201010057}

\bibitem{QuirogaBarrancoSanchezNungaray2021}
Quiroga-Barranco, R.;
S\'{a}nchez-Nungaray, A.:
Moment maps of Abelian groups and commuting Toeplitz operators acting on the unit ball.
J. Funct. Anal. \textbf{281}, 109039 (2021).
\doi{10.1016/j.jfa.2021.109039}


\bibitem{RamirezOrtegaSanchezNungaray2015}
Ram\'{i}rez Ortega, J.; S\'{a}nchez-Nungaray, A.:
Toeplitz operators with vertical symbols acting on the poly-Bergman spaces of the upper half-plane.
Complex Anal. Oper. Theory \textbf{9}, 1801--1817  (2015).
\doi{10.1007/s11785-015-0469-4}

\bibitem{Sakai1971}
Sakai, S.:
C*-Algebras and W*-algebras.
Springer-Verlag,
Berlin, Heidelberg, New York
(1971).

\bibitem{SteinwartHushScovel2006}
Steinwart, I.;
Hush, D.;
Scovel, C. (2006):
An explicit description of the reproducing kernel Hilbert spaces of Gaussian RBF kernels,
in IEEE Transactions on Information Theory, vol. 52, no. 10, pp. 4635--4643.
\doi{10.1109/TIT.2006.881713}

\bibitem{Suarez2008}
Su\'{a}rez, D.:
The eigenvalues of limits of radial Toeplitz operators.
Bull. Lond. Math. Soc. \textbf{40}, 631--641  (2008).
\doi{10.1112/blms/bdn042}

\bibitem{Takesaki2002}
Takesaki, M.:
Theory of Operator Algebras I.
2nd printing of the 1st edition.
Springer (2002).

\bibitem{Vasilevski1999BergmanToeplitz}
Vasilevski, N.L.:
On Bergman-Toeplitz operators with commutative symbol algebras.
Integr. Equ. Oper. Theory \textbf{34}, 107--126 (1999).
\doi{10.1007/BF01332495}

\bibitem{Vasilevski2008book}
Vasilevski, N.L.:
Commutative Algebras of Toeplitz Operators on the Bergman Space.
Birkh\"{a}user, Basel (2008).
\doi{10.1007/978-3-7643-8726-6}

\bibitem{Xia2015}
Xia, J.:
Localization and the Toeplitz algebra on the Bergman space.
J. Funct. Anal. \textbf{269}, 781--814  (2015).
\doi{10.1016/j.jfa.2015.04.011}

\end{thebibliography}
\end{document}